\documentclass[12pt]{smfart}
\usepackage[french]{babel}
\usepackage{amscd,verbatim,amssymb}
\usepackage{color}

\newcommand{\A}{\mathbf{A}}

\newcommand{\C}{\mathbf{C}}
\newcommand{\F}{\mathbb{F}}

\renewcommand{\H}{\mathbb{H}}
\renewcommand{\L}{\mathbb{L}}

\newcommand{\Q}{\mathbf{Q}}
\newcommand{\Z}{\mathbf{Z}}

\newcommand{\sA}{\mathcal{A}}

\newcommand{\sH}{\mathcal{H}}
\newcommand{\sK}{\mathcal{K}}
\newcommand{\sO}{\mathcal{O}}

\newcommand{\sX}{\mathcal{X}}

\newcommand{\et}{{\operatorname{\acute{e}t}}}
\newcommand{\Zar}{{\operatorname{Zar}}}
\newcommand{\Nis}{{\operatorname{Nis}}}

\newcommand{\Ab}{\operatorname{Ab}}
\newcommand{\Div}{\operatorname{Div}}

\newcommand{\Chow}{{\operatorname{Chow}}}

\newcommand{\AJ}{{\operatorname{AJ}}}

\newcommand{\tr}{{\operatorname{tr}}}

\newcommand{\alg}{{\operatorname{alg}}}
\renewcommand{\hom}{{\operatorname{hom}}}
\newcommand{\num}{{\operatorname{num}}}
\newcommand{\cont}{{\operatorname{cont}}}
\newcommand{\car}{\operatorname{car}}
\newcommand{\Spec}{\operatorname{Spec}}
\newcommand{\Pic}{\operatorname{Pic}}

\newcommand{\Hom}{\operatorname{Hom}}

\newcommand{\codim}{\operatorname{codim}}

\newcommand{\tors}{{\operatorname{tors}}}
\newcommand{\tate}{{\operatorname{Tate}}}

\newcommand{\op}{{\operatorname{op}}}
\newcommand{\cl}{\operatorname{cl}}
\newcommand{\corg}{\operatorname{corang}}
\newcommand{\Coker}{\operatorname{Coker}}
\newcommand{\Ker}{\operatorname{Ker}}
\newcommand{\IM}{\operatorname{Im}}
\newcommand{\Griff}{\operatorname{Griff}}

\newcommand{\s}{\scriptstyle}

\newcommand{\oo}{\operatornamewithlimits{\otimes}}
\newcommand{\inj}{\hookrightarrow}
\newcommand{\Inj}{\lhook\joinrel\longrightarrow}
\newcommand{\surj}{\rightarrow\!\!\!\!\!\rightarrow}
\newcommand{\Surj}{\relbar\joinrel\surj}
\newcommand{\by}[1]{\overset{#1}{\longrightarrow}}
\newcommand{\iso}{\by{\sim}}
\newcommand{\yb}[1]{\overset{#1}{\longleftarrow}}
\newcommand{\osi}{\yb{\sim}}
\renewcommand{\lim}{\varprojlim}
\newcommand{\colim}{\varinjlim}

\newcounter{spec}
\newenvironment{thlist}{\begin{list}{\rm{(\roman{spec})}}%
{\usecounter{spec}\labelwidth=20pt\itemindent=0pt\labelsep=10pt}}%
{\end{list}}%

\newtheorem{thm}{Th\'eor\`eme}[section]
\newtheorem{lemme}[thm]{Lemme}
\newtheorem{lemmecle}[thm]{Lemme-cl\'e}
\newtheorem{prop}[thm]{Proposition}
\newtheorem{cor}[thm]{Corollaire}

\theoremstyle{definition}
\newtheorem{defn}[thm]{D\'efinition}
\theoremstyle{remark}
\newtheorem{rque}[thm]{Remarque}
\newtheorem{rques}[thm]{Remarques}

\newtheorem{conv}[thm]{Convention}

\numberwithin{equation}{section}

\begin{document}

\title{Classes de cycles motiviques \'etales}
\author{Bruno Kahn}
\address{Institut de Math\'ematiques de Jussieu\\UMR 7586\\ Case 247\\4 place
Jussieu\\75252 Paris Cedex 05\\France}
\email{kahn@math.jussieu.fr}
\date{11 juin 2011}
\begin{altabstract} Let $X$ be a smooth variety over a field $k$, and $l$ be a prime number. We
construct an exact sequence
\[0\to H^0(X,\sH^3_\cont(\Z_l(2)))\otimes \Q/\Z\to H^0(X,\sH^3_\et(\Q_l/\Z_l(2)))\to
C_\tors\to 0\]
where $\sH^i_\et(\Q_l/\Z_l(2))$ and $\sH^i_\cont(\Z_l(2))$  are the Zariski sheaves associated to \'etale and continuous \'etale cohomology and $C$ is the cokernel of Jannsen's $l$-adic cycle class on
$CH^2(X)\otimes \Z_l$ if $l\ne \text{char\ } k$ or a variant of it if $l=\text{char\ } k$. If
$k=\C$, this gives another proof of a theorem of Colliot-Th\'el\`ene--Voisin, avoiding a
recourse to the Bloch-Kato conjecture in degree $3$. If $k$ is separably closed and $l\ne \text{char\ } k$, still in the spirit of Colliot-Th\'el\`ene and Voisin we get an exact sequence
\begin{multline*}
0\to \Griff^2(X,\Z_l)_\tors\to H^3_\tr(X,\Z_l(2))\otimes \Q/\Z\\
\to H^0(X,\sH^3_\cont(\Z_l(2)))\otimes \Q/\Z\to \Griff^2(X,\Z_l)\otimes \Q/\Z\to 0
\end{multline*}
where $H^3_\tr(X,\Z_l(2))$ is the quotient of $l$-adic cohomology by the first step of the
coniveau filtration and $\Griff^2(X,\Z_l)$ is an $l$-adic Griffiths group. 

If $k$ is the algebraic closure of a finite field $k_0$ and $X$ is ``of abelian type" and verifies the Tate conjecture, $\Griff^2(X,\Z_l)$ is torsion and
$H^0(X,\sH^3_\et(\Q_l/\Z_l(2)))$ is finite provided $H^3_\tr(X,\Q_l(2))=0$. On the other hand,
a theorem of Schoen gives an example where $H^0(X,\sH^3_\et(\Q_l/\Z_l(2)))$ is finite but
\allowbreak $H^3_\tr(X,\Q_l(2))\ne 0$.
\end{altabstract}
\maketitle

\tableofcontents

\section{Introduction} Soient $k$ un corps, $X$ une $k$-vari\'et\'e lisse et $l$ un nombre
premier diff\'erent de $\car k$. Uwe Jannsen a d\'efini une classe de cycle $l$-adique
\begin{equation}\label{eqjan}
CH^2(X)\otimes \Z_l\by{\cl^2} H^4_\cont(X,\Z_l(2))
\end{equation}
\`a valeurs dans sa cohomologie \'etale continue \cite[Lemma 6.14]{jannsen0}. En i\-mi\-tant sa
construction \`a partir d'un th\'eor\`eme de Geisser et Levine \cite{gl2}, on obtient une
variante
$p$-adique de \eqref{eqjan} si $k$ est de caract\'eristique $p>0$, o\`u le second membre est 
une version continue de la cohomologie de Hodge-Witt logarithmique. Notons
$\sH^3_\et(\Q_l/\Z_l(2)))$ (resp.
$\sH^3_\cont(\Z_l(2)))$) le faisceau Zariski associ\'e au pr\'efaisceau $U\mapsto
H^3_\et(U,\Q_l/\Z_l(2))$ (resp. $U\mapsto H^3_\cont(U,\Z_l(2))$).  Le but de cet article est de
d\'emontrer l'analogue $l$-adique d'un th\'eor\`eme de Jean-Louis Colliot-Th\'el\`ene et Claire
Voisin
\cite[th. 3.6]{ct-v}:

\begin{thm} \label{p2et} Soit $l$ un nombre premier quelconque et soit $C$ le conoyau de
\eqref{eqjan}. On a  une suite exacte
\[0\to H^0(X,\sH^3_\cont(\Z_l(2)))\otimes \Q/\Z\by{f} H^0(X,\sH^3_\et(\Q_l/\Z_l(2)))\by{g}
C_\tors\to 0.\]
\end{thm}

Noter que $C_\tors$ est fini si
$H^4_\et(X,\Z_l(2))$ est un
$\Z_l$-module de type fini: ceci se produit pour $l\ne p$ si $k$ est fini, ou plus
g\'en\'eralement si les groupes de cohomologie galoisienne de
$k$ \`a coefficients finis sont finis, par exemple (comme me l'a fait remarquer J.-L.
Colliot-Th\'el\`ene) si $k$ est un corps local [sup\'erieur] ou un corps s\'eparablement clos.
Dans ces cas, le th\'eor\`eme \ref{p2et} implique donc que $H^0(X,\sH^3_\et(\Q_l/\Z_l(2)))$ est
extension d'un groupe fini par un groupe divisible. (Pour $l=p$, $C_\tors$ est fini si $k$ est
fini et $X$ projective d'apr\`es \cite[p. 589, prop. 4.18]{gros-suwa}.)

Lorsque $k=\C$, ceci donne une autre d\'emonstration du th\'eor\`eme de
Colliot-Th\'el\`ene--Voisin en utilisant l'isomorphisme de comparaison entre cohomologies de
Betti et $l$-adique et sa compatibilit\'e aux classes de cycles respectives. 

La d\'emonstration de \cite[th. 3.6]{ct-v}  donn\'ee par Colliot-Th\'el\`ene et Voisin utilise
l'exactitude du complexe de faisceaux Zariski de cohomologie de Betti
\[0\to \sH^3_\cont(\Z(2))\to \sH^3_\cont(\Q(2))\to \sH^3_\et(\Q/\Z(2))\to 0.\]

Son exactitude \`a gauche d\'ecoule du th\'eor\`eme de Merkurjev-Suslin, c'est-\`a-dire la
conjecture de Bloch-Kato en degr\'e $2$; celle \`a droite d\'ecoule de la conjecture de
Bloch-Kato en degr\'e
$3$, dont la d\'emonstration a \'et\'e conclue r\'ecemment par Voevodsky et un certain nombre
d'auteurs. 

La d\'emonstration donn\'ee ici \'evite le recours \`a cette derni\`ere conjecture: elle ne
repose que sur le th\'eor\`eme de Merkurjev-Suslin plus un formalisme triangul\'e un peu
sophistiqu\'e, mais dont, je pense, la sophistication est bien inf\'erieure aux ingr\'edients
de la preuve du th\'eor\`eme de Voevodsky et al.
 
 Son principe est le suivant. La classe de cycle \eqref{eqjan} se
prolonge en une classe ``\'etale"
\begin{equation}\label{eqet}
H^4_\et(X,\Z(2))\otimes \Z_l\to H^4_\et(X,\Z_l(2))
\end{equation}
o\`u le terme de gauche est un groupe de cohomologie motivique \'etale de $X$. Le
th\'eor\`eme de comparaison de la cohomologie motivique \'etale \`a coefficients
finis avec la cohomologie \'etale des racines de l'unit\'e tordues ou de Hodge-Witt
logarithmique (th\'eor\`eme
\ref{t1} a) et b)) implique que \eqref{eqet} est de noyau divisible et de conoyau sans torsion.
On en d\'eduit une surjection $g$ de noyau divisible dans le th\'eor\`eme \ref{p2et} 
\`a l'aide de la suite exacte
\[
0\to CH^2(X)\to H^4_\et(X,\Z(2))\to H^0(X,\sH^3_\et(\Q/\Z(2)))\to 0.
\]
qui est rappel\'ee/\'etablie dans la proposition \ref{p1}. Ceci est fait au \S \ref{3.2}. La
d\'etermination du noyau est plus technique et je renvoie au \S \ref{s3.2} pour les d\'etails.

Pour justifier la structure du noyau et du conoyau de \eqref{eqet}, il faut consid\'erer le
``c\^one" de l'application classe de cycle: ceci est expliqu\'e au \S \ref{ladique2}. 

On obtient de plus des renseignements suppl\'ementaires sur le groupe
$H^0(X,\sH^3_\cont(\Z_l(2)))\otimes \Q/\Z$:

\begin{enumerate}
\item Si $k$ est fini et $X$ projective lisse, dans la classe $B_\tate(k)$ de \cite[D\'ef. 1
b)]{cell} ce groupe est nul (\S\ref{fini}).
\item Si $k$ est s\'eparablement clos et $l\ne \car k$, toujours dans l'esprit de \cite{ct-v}
on a une suite exacte (cf. corollaire \ref{c3.2}):
\[0\to H^3_\tr(X,\Z_l(2))\to H^0(X,\sH^3_\cont(\Z_l(2)))\to \Griff^2(X,\Z_l)\to 0\]
o\`u $\Griff^2(X,\Z_l)$ est le groupe des cycles de codimension $2$ \`a coefficients
$l$-adiques, modulo l'\'equivalence alg\'ebrique, dont la classe de cohomologie $l$-adique est
nulle, et $H^3_\tr(X,\Z_l(2))$ est le quotient de $H^3_\cont(X,\Z_l(2))$ par le premier cran de
la filtration par le coniveau. 

Comme le groupe $H^0(X,\sH^3_\cont(\Z_l(2)))$ est sans torsion (lemme  \ref{l3.4}), on en
d\'eduit une suite exacte (proposition \ref{l4.2}):
\begin{multline*}
0\to \Griff^2(X,\Z_l)_\tors\to H^3_\tr(X,\Z_l(2))\otimes \Q/\Z\\
\to H^0(X,\sH^3_\cont(\Z_l(2)))\otimes \Q/\Z\to \Griff^2(X,\Z_l)\otimes \Q/\Z\to 0.
\end{multline*}
\item Si $k$ est la cl\^oture alg\'ebrique d'un corps fini $k_0$ et que $X$ provient de la
classe $B_\tate(k_0)$ de \cite{cell}, le groupe $\Griff^2(X,\Z_l)$ est de torsion et la suite
exacte ci-dessus se simplifie en:
\begin{multline*}
0\to \Griff^2(X,\Z_l)\to H^3_\tr(X,\Z_l(2))\otimes \Q/\Z\\
\to H^0(X,\sH^3_\cont(\Z_l(2)))\otimes \Q/\Z\to 0
\end{multline*}
(th\'eor\`eme \ref{t3.1}). En particulier, $H^0(X,\sH^3_\et(\Q_l/\Z_l(2))$ est fini d\`es que
$H^3_\tr(X,\Q_l(2))=0$ (ceci est conjecturalement vrai sur tout corps s\'eparablement clos, cf.
remarque \ref{r4.3}), mais un exemple de Schoen montre que cette condition n'est pas
n\'ecessaire (Proposition \ref{p5.2} et th\'eor\`eme \ref{t5.1}).
\end{enumerate}

\enlargethispage*{20pt}

\subsection*{Remerciements} Cet article est l'\'elaboration d'un texte r\'edig\'e en janvier
2010, lui-m\^eme directement inspir\'e de plusieurs discussions avec Jean-Louis
Colliot-Th\'e\-l\`e\-ne autour de son travail avec Claire Voisin \cite{ct-v}, alors en cours
d'a\-ch\`e\-ve\-ment. Il a ensuite b\'en\'efici\'e de la lecture de \cite{ct-v}, ainsi que de
discussions avec Colliot-Th\'el\`ene autour d'un article commun en projet
\cite{ctk}: je le remercie de sa lecture critique de versions pr\'eliminaires de ce texte, et
notamment  de m'avoir fait remarquer qu'il n'est pas n\'ecessaire de supposer $k$ de type fini
sur son sous-corps premier dans le th\'eor\`eme \ref{p2et}. Je remercie l'IMPA de Rio de
Janeiro pour son hospitalit\'e pendant la fin de sa r\'edaction, ainsi que la coop\'eration
franco-br\'esilienne pour son soutien. Pour finir, je remercie Luc Illusie de m'avoir indiqu\'e
une d\'emonstration (triviale!) du lemme \ref{lillusie}, lemme qui \`a ma connaissance ne figure
nulle part dans la litt\'erature.

\section{Groupes de Chow sup\'erieurs}

Cette section comporte presque exclusivement des rappels sur les grou\-pes
de Chow sup\'erieurs: le lecteur au courant peut la parcourir 
rapidement. Elle a pour but principal de fournir une preuve compl\`ete de la proposition
\ref{p1}, \'evitant le complexe $\Gamma(2)$ de Lichtenbaum.

\subsection{Groupes de Chow sup\'erieurs}\label{2.1} Soit $k$ un corps.
Dans \cite{bloch}, Bloch associe \`a tout $k$-sch\'ema alg\'ebrique
$X$ des complexes de groupes ab\'eliens $z^n(X,\bullet)$ ($n\ge 0$), concentr\'es en degr\'es
(homologiques) $\ge 0$: rappelons qu'on pose $\Delta^p= \Spec k[t_0,\dots,t_p]/(\sum t_i-1)$,
que $z^n(X,p)$ est le groupe ab\'elien libre sur les ferm\'es int\`egres de codimension $n$ de
$X\times_k\Delta^p$ qui rencontrent les faces proprement, et que la
diff\'erentielle $d_p$ est obtenue comme somme altern\'ee des intersections avec les faces
de dimension $p-1$. Les groupes d'homologie
$CH^n(X,p)$ de $z^n(X,\bullet)$ sont les \emph{groupes de Chow sup\'erieurs} de $X$: on a
$CH^n(X,0)=CH^n(X)$ par construction. 

Les $z^n(X,\bullet)$ sont contravariants pour les morphismes plats, en
par\-ti\-cu\-lier \'etales; ils d\'e\-fi\-nis\-sent en fait des complexes de faisceaux sur le
petit site \'etale d'un sch\'ema lisse $X$ donn\'e. Ils sont aussi covariants pour les
morphismes propres, en particulier pour les immersions ferm\'ees.

Si $Y$ est un ferm\'e de $X$, on notera ici 
\begin{align*}
z^n_Y(X,\bullet)&=\text{Fib}\left(z^n(X,\bullet)\by{j^*}z^n(X-Y,\bullet)\right)\\
CH^n_Y(X,p)&=H_p(z^n_Y(X,\bullet))
\end{align*}
o\`u $j:X-Y\to X$ est l'immersion ouverte compl\'ementaire et Fib d\'esigne la fibre
homotopique (d\'ecal\'e du \emph{mapping cone}). Si on tensorise par un groupe ab\'elien $A$, on
\'ecrit
$CH^n_Y(X,p,A)$.

On a le th\'eor\`eme fondamental suivant, qui est une vaste g\'en\'eralisation du lemme de
d\'eplacement de Chow (Bloch, \cite[Th. 3.1 et 4.1]{bloch}, preuves corrig\'ees dans
\cite{bloch2}):

\begin{thm}\label{tloc} a) Les groupes de Chow sup\'erieurs sont contravariants pour les
morphismes quelconques de but lisse entre vari\'et\'es quasi-projectives. Ils commutent aux limites projectives filtrantes \`a morphismes de transition affines.\\
b) Soient $X$ un $k$-sch\'ema
quasi-projectif
\'equ\-idi\-men\-sion\-nel,  $i:Y\to X$ un ferm\'e \'equidimensionnel et $j:U\to X$ l'ouvert
com\-pl\'e\-men\-tai\-re. Soit $d$ la codimension de $Z$ dans $X$. Alors le morphisme naturel
\[z^{n-d}(Y,\bullet)\by{i_*} z^n_Y(X,\bullet)\]
est un quasi-isomorphisme.\\
c) On dispose de produits d'intersection
\begin{equation}\label{eqint}
CH^m(X,p)\times CH^n(X,q)\to CH^{m+n}(X,p+q)
\end{equation}
pour $X$ quasi-projectif lisse.
\end{thm}

De la partie b) de ce th\'eor\`eme, on d\'eduit que pour tout groupe ab\'elien $A$, la th\'eorie
cohomologique
\`a supports
$(X,Z)\mapsto CH^n_Z(X,\bullet, A)$ d\'efinie sur la cat\'egorie des $k$-sch\'emas
quasi-projectifs lisses
\cite[def. 5.1.1 a)]{cthk} v\'erifie l'axiome {\bf COH1} de loc. cit., p. 53. D'apr\`es loc.
cit., th. 7.5.2, on a donc des isomorphismes pour tout $(n,p)$
\[CH^n(X,p,A)\iso H^{-p}_\Zar(X, z^n(-,\bullet)\otimes A)\iso H^{-p}_\Nis(X,
z^n(-,\bullet)\otimes A)\] 
pour $X$ quasi-projectif lisse. 

Si $X$ est seulement lisse, le second isomorphisme persiste.

\begin{defn}\label{d1} Soit $X$ un $k$-sch\'ema lisse (essentiellement de type fini), et soit
$\tau$ une topologie de Grothendieck moins fine que la topologie \'etale sur la cat\'egorie des
$k$-sch\'emas lisses de type fini: en pratique $\tau\in \{\Zar,\Nis,\et\}$. On note $A_X(n)_\tau$ le
complexe de faisceaux $z^n(-,\bullet)\otimes A[-2n]$ sur $X_\tau$ et 
$H^*_\tau(X,A(n))$ l'hypercohomologie de $X$ \`a coefficients dans le complexe $A_X(n)$ (pour
la topologie $\tau$). C'est la \emph{cohomologie motivique de poids $n$ \`a coefficients dans
$A$} pour la topologie concern\'ee. \\
Pour simplifier, on supprime $\tau$ de la notation si
$\tau=\Zar$ ou $\Nis$ (voir ci-dessus).
\end{defn}

On a donc un isomorphisme, pour $X$ quasi-projectif lisse: 
\begin{equation}\label{eqiso}
CH^n(X,2n-i)\iso H^i(X, A(n)).
\end{equation}

On a
\begin{align}
\Z_X(0) &=\Z\label{eqz0}\\
\Z_X(1)&\simeq \sO_X^*[-1].\label{eqz1}
\end{align}
(\cite[Cor. 6.4]{bloch} pour le second quasi-isomorphisme).

L'isomorphisme \eqref{eqz1} se g\'en\'eralise ainsi:

\begin{thm}[Nesterenko-Suslin, Totaro]\label{tK} Supposons que $X=\Spec K$, o\`u $K$ est un corps.
L'isomorphisme \eqref{eqz1} et les produits \eqref{eqint} induisent des isomorphismes
\begin{align*}
K_n^M(K)&\iso H^n(K,\Z(n))\\
K_n^M(K)/m&\iso H^n(K,\Z/m(n))
\end{align*}
pour $m>0$, o\`u $K_n^M$ d\'esigne la $K$-th\'eorie de Milnor..
\end{thm}

\begin{proof} Voir \cite{ns} ou \cite{totaro} pour le premier \'enonc\'e; le second s'en
d\'eduit puisque $H^{n+1}(K,\Z(n))=0$.
\end{proof}

\begin{rque}\label{r1} L'isomorphisme \eqref{eqiso} vaut pour
$i\ge 2n$, m\^eme si $X$ n'est pas quasi-projectif. En effet, le terme de droite est
l'aboutissement d'une suite spectrale de coniveau \cite[Rk. 5.1.3 (3)]{cthk} qui, gr\^ace au
th\'eor\`eme
\ref{tloc}, prend la forme suivante:
\[E_1^{p,q}= \bigoplus_{x\in X^{(p)}} H^{q-p}(k(x),A(n-p))\Rightarrow H^{p+q}(X,A(n)).\]

On a $H^i(F,A(r))=0$ pour $i>r$ et
tout corps $F$, car $A_F(r)$ est un complexe concentr\'e en degr\'es $\le r$. On en d\'eduit
d\'ej\`a que
$H^i(X,A(n))=0$ pour $i>2n$, et on a \'evidemment
$CH^n(X,p,A)=0$ pour $p<0$. Quant \`a
$H^{2n}(X,A(n))$, il s'ins\`ere dans une suite exacte
\[E_1^{n-1,n}\by{d_1} E_1^{n,n}\to H^{2n}(X,A(n))\to 0\]
qui s'identifie \`a la suite exacte
\[\bigoplus_{x\in X_{(1)}} k(x)^*\otimes A\by{\Div} \bigoplus_{x\in X_{(0)}} A\to
CH^n(X)\otimes A\to 0\] via \eqref{eqz0} et \eqref{eqz1} (l'identification de la
diff\'erentielle
$d_1$ \`a l'application diviseurs est facile \`a partir du th\'eor\`eme \ref{tloc} appliqu\'e
pour $n=1$). D'autre part, on calcule ai\-s\'e\-ment que $CH^n(X,0,A)=CH^n(X)\otimes A$ sans
supposer
$X$ quasi-projectif.
\end{rque}

Le lemme suivant raffine une partie de la remarque \ref{r1}: sa d\'e\-mons\-tra\-tion est moins
\'el\'ementaire.

\begin{lemme}\label{l2.1} Le complexe $\Z_X(n)$ est concentr\'e en degr\'es $\le n$: autrement
dit,
$\sH^i(\Z_X(n))=0$ pour $i>n$.\qed
\end{lemme}

\begin{proof} La th\'eorie cohomologique \`a supports 
\[(X,Y)\mapsto H^*_Y(X,\Z(n))\] 
v\'erifie les axiomes {\bf COH1} et {\bf COH3} de \cite{cthk}: le premier, ``excision
\'etale", r\'esulte facilement du th\'eor\`eme \ref{tloc} b) et le second, invariance par
homotopie, est d\'emontr\'e dans \cite[th. 2.1]{bloch}. Il r\'esulte alors de \cite[cor.
5.1.11]{cthk} qu'elle v\'erifie la conjecture de Gersten (c'est d\'ej\`a d\'emontr\'e dans
\cite[th. 10.1]{bloch}). En particulier, les faisceaux $\sH^i(\Z(n))$ s'injectent dans leur
fibre g\'en\'erique, et on est r\'eduit au cas \'evident d'un corps de base.
\end{proof}

\subsection{Comparaisons} \`A partir de maintenant, $X$ d\'esigne un $k$-sch\'ema lisse.

\begin{thm}\label{t1} a)  Si $m$ est inversible dans $k$, il existe un
quasi-isomorphisme canonique 
\[(\Z/m)_X(n)_\et\iso \mu_m^{\otimes n}.\]
b) Si $k$ est de caract\'eristique $p>0$ et $r\ge 1$, il existe
un quasi-iso\-mor\-phis\-me canonique
\[(\Z/p^r)_X(n)_\et\iso \nu_r(n)[-n]\]
o\`u $\nu_r(n)$ est le $n$-\`eme faisceau de Hodge-Witt logarithmique.\\
c) Soit $\alpha$ la projection de $X_\et$ sur $X_\Zar$. Alors la fl\`eche d'adjonction
\[\Q_X(n)\to R\alpha_*\Q_X(n)_\et\]
est un isomorphisme.
\end{thm}

\begin{proof} a) et b) sont dus \`a Geisser-Levine: a) est \cite[th. 1.5]{gl} et b) est
\cite[th. 8.5]{gl2} (si $k$ est parfait: voir le corollaire \ref{cgl} en
g\'en\'eral)\footnote{Cette derni\`ere r\'ef\'erence indique que l'hypoth\`ese ``$X$ lisse sur
$k$" devrait \^etre remplac\'ee par ``$X$ r\'egulier de type fini sur $k$" dans une grande
partie de ce texte.}. c) est un fait g\'en\'eral pour un complexe de faisceaux Zariski
$C$ de
$\Q$-espaces vectoriels sur un sch\'ema normal $S$: on se ram\`ene au cas o\`u $S$ est local et
o\`u $C = A[0]$ est un faisceau concentr\'e en degr\'e $0$. Alors $A$ est un faisceau constant
de
$\Q$-espaces vectoriels et cela r\'esulte de \cite[th. 2.1]{deninger}.
\end{proof}

\begin{defn}\label{d2.1} Soit $n\ge 0$. Pour $l$ premier diff\'erent de $\car k$, on note
\[\Q_l/\Z_l(n)=\colim_r \mu_{l^r}^{\otimes n}.\]
Pour $l=\car k$, on note
\[\Q_l/\Z_l(n)=\colim_r \nu_r(n)[-n].\]
Enfin, on note
\[\Q/\Z(n) = \bigoplus_l \Q_l/\Z_l(n).\]
C'est un objet de la cat\'egorie d\'eriv\'ee des groupes ab\'eliens sur le gros site \'etale
de $\Spec k$. 
\end{defn}

Le th\'eor\`eme \ref{t1} montre qu'on a un isomorphisme, pour tout $X$ lisse sur $k$:
\begin{equation}\label{eq2.5}
(\Q/\Z)_X(n)_\et\iso (\Q/\Z)(n)_{|X}.
\end{equation}

Nous utiliserons cette identification dans la suite sans mention ul\-t\'e\-rieu\-re.

On a alors:

\begin{cor}\label{c2.1} Pour tout $i>n+1$, l'homomorphisme de faisceaux Zariski
\[\sH^{i-1}(R\alpha_* \Q/\Z(n))\to \sH^i(R\alpha_* \Z_X(n)_\et)\]
\'emanant du th\'eor\`eme \ref{t1} a) et b) est un isomorphisme.
\end{cor}

\begin{proof} Dans la suite exacte de faisceaux Zariski
\begin{multline*}
\sH^{i-1}(R\alpha_* \Q_X(n)_\et)\to \sH^{i-1}(R\alpha_* \Q/\Z(n))\\
\to \sH^i(R\alpha_* \Z_X(n)_\et)\to \sH^i(R\alpha_* \Q_X(n)_\et) 
\end{multline*}
les deux termes extr\^emes sont nuls d'apr\`es le th\'eor\`eme \ref{t1} c) et le lemme
\ref{l2.1}.
\end{proof}

\subsection{Cohomologie \'etale de complexes non born\'es}\label{nb} Si $X$ est un sch\'ema de
dimension cohomologique \'etale \emph{a priori} non finie et si $C$ est un complexe de
faisceaux \'etales sur $X$, non born\'e inf\'erieurement, la consid\'eration de
l'hypercohomologie $H^*_\et(X,C)$ soul\`eve au moins trois difficult\'es:

\begin{enumerate}
\item une d\'efinition en forme;
\item la commutation aux limites;
\item la convergence de la suite spectrale d'hypercohomologie.
\end{enumerate}

Le premier probl\`eme est maintenant bien compris: il faut prendre une r\'esolution
K-injective, ou fibrante, de $C$, cf. par exemple Spaltenstein \cite[Th. 4.5 et Rem.
4.6]{spaltenstein}.

Le second et le troisi\`eme probl\`emes sont plus d\'elicats. Dans le cas de $\Z(n)$, le second
et implicitement le troisi\`eme est r\'esolu dans \cite[\S B.3 p. 1114]{kahnbbki} (pour la
cohomologie motivique de Suslin-Voevodsky). Rappelons l'argument: en utilisant le th\'eor\`eme
\ref{t1}, on peut ins\'erer $R\alpha_*\Z_X(n)_\et$ dans un triangle exact
\[R\alpha_*\Z_X(n)_\et\to \Q_X(n)\to R\alpha_*(\Q/\Z)_X(n)_\et\by{+1}.\]

Si $X$ est de dimension de Krull finie, l'hypercohomologie Zariski du second terme se comporte
bien, et celle du troisi\`eme terme aussi puisque c'est l'hypercohomologie \'etale d'un
complexe born\'e.

\subsection{Conjecture de Beilinson-Lichtenbaum} Cette conjecture concerne la comparaison entre
$H^*(X,A(n))$ et $H^*_\et(X,A(n))$, pour $A=\Z/m$, cf. \cite[th. 1.6]{gl}. Si $m$ est une
puissance d'un nombre premier $l\ne \car k$,  elle est
\'equivalente d'apr\`es Geisser-Levine 
\cite{gl} \`a la conjecture de Bloch-Kato en poids $n$ (pour le nombre premier $l$); donc en
poids $2$, au th\'eor\`eme de Merkurjev-Suslin. Sur un corps de caract\'eristique z\'ero, ceci
avait \'et\'e ant\'erieurement d\'emontr\'e par Suslin-Voevodsky \cite{suvo}.

De plus, pour $l=\car k$, une version de cette conjecture est d\'emontr\'ee par Geisser et
Levine dans
\cite{gl2}, cf. th\'eor\`eme \ref{tgl}. En ajoutant \`a tout ceci le th\'eor\`eme \ref{t1} c),
la conjecture de Beilinson-Lichtenbaum en poids
$n$ se retraduit en un triangle exact \cite[Th. 6.6]{voemihes}
\begin{equation}\label{eqbl}
\Z_X(n)\to  R\alpha_*\Z_X(n)_\et\to \tau_{\ge n+2}R\alpha_*\Z_X(n)_\et\to \Z_X(n)[1].
\end{equation}

Ce triangle exact contient l'\'enonc\'e (``Hilbert 90 en poids $n$"):
\[\sH^{n+1}(R\alpha_*\Z_X(n)_\et)=0.\]

\subsection{Une suite exacte} \label{2.4}

\begin{prop} \label{p1} Notons $\sH^3_\et(\Q/\Z(2))$ le faisceau Zariski associ\'e au
pr\'efaisceau $U\mapsto H^3_\et(U,\Q/\Z(2)$. Pour toute $k$-vari\'et\'e lisse $X$, on a une
suite exacte courte:
\begin{equation}\label{eqlk}
0\to CH^2(X)\to H^4_\et(X,\Z(2))\to H^0(X,\sH^3_\et(\Q/\Z(2)))\to 0.
\end{equation}
\end{prop}

\begin{proof} 
En prenant l'hypercohomologie de Zariski de $X$ \`a valeurs dans le triangle \eqref{eqbl} pour
$n=2$, on trouve une suite exacte
\begin{multline*}
0\to H^4(X,\Z(2))\to H^4_\et(X,\Z(2))\\
\to H^0(X,R^4\alpha_*\Z(2)_\et)\to H^5(X,\Z(2)). 
\end{multline*}

D'apr\`es la remarque \ref{r1}, on a $H^4(X,\Z(2))=CH^2(X)$ et $H^5(X,\Z(2))\allowbreak=0$.
D'autre part, le triangle exact
\[\Z_X(2)_\et\to \Q_X(2)_\et\to (\Q/\Z)_X(2)_\et\]
provenant du th\'eor\`eme \ref{t1} donne une longue suite exacte de faisceaux
\[R^3\alpha_*\Q_X(2)_\et \to R^3\alpha_*(\Q/\Z)_X(2)_\et\to R^4\alpha_*\Z_X(2)_\et\to
R^4\alpha_*\Q_X(2)_\et.\]

On a $R^3\alpha_*\Q_X(2)_\et=R^4\alpha_*\Q_X(2)_\et=0$ puisque $\Z_X(2)$ est concentr\'e en
degr\'es $\le 2$, cf. th\'eor\`eme \ref{t1} c). Ce qui conclut, via l'isomorphisme
\eqref{eq2.5}.
\end{proof}

\begin{rques}\label{r2.1} 1) La suite exacte \eqref{eqlk} appara\^\i t dans  \cite[Th. 1.1,
\'eq. (9)]{k}, avec $\Z(2)$ remplac\'e par le complexe de Lichtenbaum $\Gamma(2)$; \`a la
$2$-torsion pr\`es, elle est d\'ej\`a chez Lichtenbaum \cite[Th. 2.13 et rem. 2.14]{licht}. Il
est probable qu'on a un isomorphisme
\begin{equation}\label{eqtrunc}
\Gamma(2, X)\simeq \tau_{\ge 1}\left(z^2(X,\bullet)[-4]\right)
\end{equation}
dans $D(X_\Zar)$ pour tout $k$-sch\'ema lisse $X$.\footnote{Par
ailleurs, la conjecture de Beilinson-Soul\'e pr\'edit que $\Z(2)\to \tau_{\ge 1}\Z(2)$ est un
quasi-isomorphisme, mais elle n'a pas d'importance pour ce travail.} Une fonctorialit\'e
suffisante de cet isomorphisme impliquerait qu'il peut s'\'etalifier. Dans \cite[Th. 7.2]{bl},
un isomorphisme \eqref{eqtrunc} est construit pour  $X=\Spec k$. Mais
\eqref{eqtrunc} ne semble pas appara\^\i tre dans la litt\'erature en g\'en\'eral. 

2) En se reposant sur la conjecture de Bloch-Kato en poids $3$, on obtient de la
m\^eme mani\`ere une suite exacte
\begin{multline*}
0\to H^2(X,\sK_3)\to H^5_\et(X,\Z(3))\to
H^0(X,\sH^4_\et(\Q/\Z(3))\\
\to CH^3(X)\to H^6_\et(X,\Z(3)).
\end{multline*}

Cette suite appara\^\i t dans \cite[Rem. 4.10]{cell}, sauf que le premier terme est
$H^5(X,\Z(3))$; son identification avec $H^2(X,\sK_3)$ se fait \`a l'aide de la suite spectrale
de coniveau de la remarque \ref{r1}.
\end{rques}

\subsection{D'autres suites exactes}

\begin{prop}\label{p3} On a des suites exactes
\begin{multline*}
0\to H^3(X,\Z/m(2))\to H^3_\et(X,\Z/m(2))\to H^0(X,\sH^3_\et(\Z/m(2)))\\
\to CH^2(X)\otimes \Z/m\to H^4_\et(X,\Z/m(2))
\end{multline*}
($m>0$),
\begin{multline*}
0\to H^3(X,\Q/\Z(2))\to H^3_\et(X,\Q/\Z(2))\to H^0(X,\sH^3_\et(\Q/\Z(2)))\\
\to CH^2(X)\otimes \Q/\Z\to H^4_\et(X,\Q/\Z(2)).
\end{multline*}
\end{prop}

\begin{proof} Elle s'obtiennent comme dans la preuve de la proposition \ref{p1} en prenant la
cohomologie des triangles exacts
\[\Z/m(2)\to R\alpha_* (\Z/m)_\et(2)\to \tau_{\ge 3}R\alpha_* (\Z/m)_\et(2)\by{+1}\] 
\[\Q/\Z(2)\to R\alpha_* (\Q/\Z)_\et(2)\to \tau_{\ge 3}R\alpha_* (\Q/\Z)_\et(2)\by{+1}\] 
obtenus en tensorisant \eqref{eqbl} par $\Z/m$ ou $\Q/\Z$ au sens d\'eriv\'e. 
\end{proof}

On reconna\^\i t donc dans $H^3(X,\Z/m(2))$ le groupe $NH^3_\et(X,\Z/m(2))$ de Suslin \cite[\S
4]{suslin}. On peut sans doute montrer que la seconde suite exacte co\"\i ncide avec celle de
\cite[p. 790, rem. 2]{ctss}. 

\section{Cohomologie $l$-adique et $p$-adique}

Dans cette section, $k$ est un corps quelconque, de caract\'eristique $p\ge 0$.

\subsection{Classe de cycle $l$-adique et $p$-adique}\label{ladique2} Soit $l$ un nombre premier
quelconque. Pour toute
$k$-vari\'et\'e lisse $X$, on a des applications ``classe de cycle
$l$-adique"
\begin{equation}\label{eq1}
H^i_\et(X,\Z(n))\otimes \Z_l\to H^i_\cont(X,\Z_l(n)).
\end{equation}

Ces homomorphismes proviennent d'un morphisme de complexes (dans la cat\'e\-go\-rie
d\'eriv\'ee de la cat\'egorie des complexes de faisceaux sur $X_\et$)
\begin{equation}\label{eq0l}
\Z_X(n)_\et\oo^L \Z_l\to \Z_l(n)_X^c
\end{equation}
o\`u
\[
\Z_l(n)_X^c=
\begin{cases}
R\lim \mu_{l^r}^{\otimes n}&\text{si $l\ne p$}\\
R\lim \nu_{r}(n)[-n]&\text{si $l= p$.}
\end{cases}
\]  

Cette construction est d\'ecrite dans \cite[\S 1.4, en part. (1.8)]{glr} pour $l\ne p$ et dans
\cite[\S 3.5]{cell} pour $l=p$. Elle repose sur celles de Geisser-Levine aux crans finis dans
\cite{gl} pour $l\ne p$ et dans \cite[d\'em. du th. 8.3]{gl2} pour $l=p$.

\begin{rque} Pour $l\ne p$ et $i=2n$, la compos\'ee de \eqref{eq1} avec l'homomorphisme
$CH^n(X)\otimes \Z_l\to H^{2n}_\et(X,\Z(n))\otimes \Z_l$ \eqref{eqlk} n'est autre que la classe
de cycle de Jannsen \cite[Lemma 6.14]{jannsen0}: cela r\'esulte de la construction m\^eme dans
\cite{gl} de l'isomorphisme du th\'eor\`eme \ref{t1} a). Pour $l=p$, il est moins clair que
\eqref{eq1} soit compatible avec la classe de cycle de Gros \cite[p. 50, d\'ef. 4.1.7 et
p. 55, prop. 4.2.33]{gros}. Cela doit pouvoir se v\'erifier directement; comme je n'en aurai
pas besoin, je laisse cet ``exercice" aux lecteurs int\'eress\'es.
\end{rque}

Notons les isomorphismes \'evidents:
\begin{equation}\label{eq3.4}
\Z_l(n)_X^c\oo^L \Z/l^r\iso
\begin{cases}
 \mu_{l^r}^{\otimes n}&\text{si $l\ne p$}\\
\nu_{r}(n)[-n]&\text{si $l= p$.}
\end{cases}
\end{equation}

\begin{defn}\label{d3.1} On note respectivement $K_X(n)_\et$ et $K_X(n)$ le choix d'un c\^one
du morphisme \eqref{eq0l} et du morphisme compos\'e
\[\Z_X(n)\oo^L \Z_l\to R\alpha_*\Z_X(n)_\et \oo^L \Z_l\to R\alpha_* \Z_l(n)_X^c\]
de sorte qu'on a un morphisme 
\[K_X(n)\to R\alpha_* K_X(n)_\et\]
compatible avec le morphisme $\Z_X(n)\oo\limits^L \Z_l\to R\alpha_*\Z_X(n)_\et \oo\limits^L
\Z_l$.
\end{defn}

\begin{rque}
Rappelons que $K_X(n)$ et $K_X(n)_\et$ ne sont uniques qu'\`a isomorphisme non unique pr\`es;
le morphisme $K_X(n)\to R\alpha_* K_X(n)_\et$ n'a pas non plus d'unicit\'e particuli\`ere. En
particulier, ces choix ne sont fonctoriels en $X$ que pour les immersions ouvertes: cela
suffira pour nos besoins ici. Toutefois, on pourrait faire des choix plus rigides (fonctoriels
pour les morphismes quelconques entre sch\'emas lisses), quitte \`a travailler dans des
cat\'egories de mod\`eles convenables.
\end{rque}

En vertu du th\'eor\`eme \ref{t1}, \eqref{eq3.4} implique im\-m\'e\-dia\-te\-ment:

\begin{prop}\label{ph1}
Le morphisme \eqref{eq0l}$\otimes \Z/l^r$ est un isomorphisme pour tout entier $r\ge 1$.
Autrement dit, les faisceaux de cohomologie de $K_X(n)_\et$ sont uniquement $l$-divisibles.
\qed
\end{prop}

\begin{cor}\label{c1a} Pour tout $(X,i,n)$, le noyau de
\eqref{eq1} est $l$-divisible et son conoyau est sans $l$-torsion.
\end{cor}

\begin{proof} On a une suite exacte
\[
H^{i-1}_\et(X,K(n))\to H^i_\et(X,\Z(n))\otimes \Z_l\to H^i_\cont(X,\Z_l(n))\to
H^i_\et(X,K(n))
\]
(o\`u $H^*_\et(X,K(n)):=\H^*_\et(X,K_X(n)_\et))$, dont les termes extr\^emes sont uniquement
divisibles.
\end{proof}

\subsection{D\'emonstration du th\'eor\`eme \ref{p2et}: premi\`ere partie}\label{3.2} On va
d\'emontrer:

\begin{prop} \label{p2.1} Soit $C$ le conoyau de \eqref{eqjan}.
On a  une surjection
\begin{equation}\label{eq2.4}
H^0(X,\sH^3_\et(\Q_l/\Z_l(2)))\Surj C_\tors
\end{equation}
de noyau divisible.
\end{prop}

\begin{proof} Utilisons la suite exacte
\eqref{eqlk}: en chassant dans le diagramme commutatif de suites exactes
\[\begin{CD}
0&\to& K @>>> CH^2(X)\otimes \Z_l @>>> H^4_\cont(X,\Z_l(2))@>>> C&\to& 0\\
&& @VVV @VVV @V{=}VV @VVV\\
0&\to& K_\et @>>> H^4_\et(X,\Z(2))\otimes \Z_l @>>> H^4_\cont(X,\Z_l(2))@>>> C_\et&\to& 0\\
\end{CD}\]
(d\'efinissant $K,K_\et$, $C$ et $C_\et$), on en d\'eduit une suite exacte
\begin{equation}\label{eq2a}
0\to K\to K_\et\to H^0(X,\sH^3_\et(\Q_l/\Z_l(2)))\to C\to C_\et\to 0.
\end{equation}
On conclut en utilisant le corollaire \ref{c1a}.
\end{proof}

\subsection{Suites spectrales de coniveau} Soit $C$ un complexe de faisceaux Zariski sur $X$.
Par une technique bien connue remontant \`a Grothendieck et Hartshorne (cf. \cite[1.1]{cthk})
on obtient une suite spectrale
\[E_1^{p,q}=\bigoplus_{x\in X^{(p)}} H^{p+q}_x(X,C)\Rightarrow H^{p+q}(X,C).\]

Cette suite spectrale est clairement naturelle en $C\in D(X_\Zar)$. On notera de mani\`ere
suggestive:
\[E_2^{p,q}=A^p(X,H^q(C))\]
de sorte qu'on a des morphismes ``edge"
\begin{equation}\label{eq3.1}
H^n(X,C)\to A^0(X,H^n(C)).
\end{equation}

Lorsque $C$ v\'erifie la conjecture de Gersten, on a des isomorphismes canoniques
\[A^p(X,H^q(C))\simeq H^p(X,\sH^q(C)).\]

D'apr\`es  \cite[cor. 5.1.11]{cthk} c'est le cas pour $C=R\alpha_*\Z_l(n)^c_X$. En effet, pour
$l\ne p$, la th\'eorie cohomologique \`a supports correspondante  v\'erifie les axiomes {\bf
COH1} (excision \'etale, on dit maintenant Nisnevich) et {\bf COH3} (invariance par homotopie)
de \cite{cthk}; pour $l=p$, elle v\'erifie {\bf COH1} et {\bf COH5}. Ce dernier axiome est la
``formule du fibr\'e projectif": il r\'esulte de \cite{gros}. Si $k$ est un corps fini, il faut
adjoindre \`a ces axiomes l'axiome {\bf COH6} de \cite[p. 64]{cthk} (existence de transferts
pour les extensions finies): il est standard.

C'est \'egalement le cas pour $C=\Z_X(n)$, cf. preuve du lemme \ref{l2.1}. Par contre ce n'est
pas clair pour $C=K_X(n)$: en effet, la r\`egle $(X,Y)\mapsto H^*_Y(X,K(n))$ ne d\'efinit pas
une th\'eorie cohomologique \`a supports sans un choix coh\'erent des c\^ones $K_X(n)$. Plus
pr\'ecis\'ement, cette r\`egle n'est pas a priori fonctorielle en $(X,Y)$ pour les morphismes
quelconques de paires.  On ne peut donc pas lui appliquer la th\'eorie de Bloch-Ogus--Gabber
d\'evelopp\'ee dans \cite{cthk}. La consid\'eration des suites spectrales de coniveau va nous
permettre de contourner ce probl\`eme. 

\subsection{Un encadrement de la cohomologie non ramifi\'ee}\label{enc} L'identification du
noyau de
\eqref{eq2.4} est plus d\'elicate.  \`A titre pr\'eparatoire, on va pousser l'analyse
du num\'ero \ref{ladique2} un peu plus loin en faisant intervenir la conjecture de Bloch-Kato
en degr\'e
$n$. 

Par l'axiome de l'octa\`edre (et la conjecture de Bloch-Kato, cf. \eqref{eqbl}), on a un diagramme commutatif de triangles distingu\'es dans
$D(X_\Zar)$, o\`u $K_X(n)$ et $K_X(n)_\et$ ont \'et\'e introduits dans la d\'efinition
\ref{d3.1}
\[\begin{CD}
\Z_X(n)\oo^L\limits \Z_l@>>> R\alpha_*\Z_l(n)_X^c@>>> K_X(n)\\
@VVV @V{||}VV @V{f}VV\\
R\alpha_*\Z_X(n)_\et\oo^L\limits \Z_l@>>> R\alpha_*\Z_l(n)_X^c@>>> R\alpha_* K_X(n)_\et\\
@VVV @VVV @VVV\\
\tau_{\ge n+2} R\alpha_*\Z_X(n)_\et\oo^L\limits \Z_l @>>> 0 @>>> C
\end{CD}\]
et o\`u $C$ est par d\'efinition ``le" c\^one de $f$. On a donc un zig-zag d'isomorphismes 
\[C\iso \tau_{\ge n+2} R\alpha_*\Z_X(n)_\et\oo^L \Z_l[1]\osi \tau_{\ge n+1}
R\alpha_*(\Q_l/\Z_l)(n)_\et\] 
o\`u l'isomorphisme de gauche provient du diagramme ci-dessus, et celui de droite provient du
corollaire \ref{c2.1}. D'o\`u un triangle exact
\[K_X(n)\to R\alpha_* K_X(n)_\et \to \tau_{\ge n+1}
R\alpha_*(\Q_l/\Z_l)(n)_\et\by{+1}.\]

Comme le deuxi\`eme terme est uniquement divisible (proposition \ref{ph1}) et que le troisi\`eme
est de torsion, cela montre que
\[K_X(n)\otimes \Q\iso R\alpha_*K_X(n)_\et.\]

On en d\'eduit:

\begin{lemme}\label{l2.4} Soit $l\ne p$. Sous la conjecture de Bloch-Kato en degr\'e $n$, le
groupe
$H^i(X,K(n))$ est uniquement divisible pour $i\le n$ et on a une suite exacte courte
\begin{multline}\label{eq2.1}
0\to H^{n+1}(X,K(n))\otimes \Q/\Z\to H^0(X,\sH^{n+1}(\Q_l/\Z_l(n)))\\
\to H^{n+2}(X,K(n))_\tors\to 0.
\end{multline}
Le m\^eme \'enonc\'e vaut pour $l=p$, en utilisant le th\'eor\`eme \ref{tgl}.\qed
\end{lemme}

Le point est maintenant d'identifier les termes extr\^emes de \eqref{eq2.1} \`a des groupes plus
concrets: nous n'y parvenons que pour $n=2$ au \S \ref{s3.2}. Mais pour $n$ quelconque, 
notons la suite exacte
\begin{multline}\label{eq2.2}
0\to \Coker\left(H^{n+2}(X,\Z(n))\otimes \Z_l\to H^{n+2}_\cont(X,\Z_l(n))\right)\\\to
H^{n+2}(X,K(n)) \to H^{n+3}(X,\Z(n))\otimes \Z_l
\end{multline}
et les homomorphismes \'evidents:
\begin{equation}\label{eq2.3}
H^{n+1}(X,K(n))\overset{\alpha}{\to} A^0(X,H^{n+1}(K(n)))\overset{\beta}{\leftarrow}
A^0(X,H^{n+1}_\cont(\Z_l(n)))
\end{equation}
o\`u $\alpha$ est l'homomorphisme \eqref{eq3.1}. 
La suite exacte \eqref{eq2.2} en induit une sur les sous-groupes de torsion. Pour $n=2$, le dernier terme est nul: on retrouve ainsi la proposition \ref{p2.1}. Pour $n>2$, la premi\`ere fl\`eche de \eqref{eq2.2} n'a plus de raison d'\^etre surjective sur la torsion.  Notons tout de m\^eme que pour $n=3$, le dernier terme de \eqref{eq2.2} n'est autre que $CH^3(X)\otimes \Z_l$ (comparer \`a la remarque \ref{r2.1}).

D'autre part:

\begin{lemme}\label{l2.2} L'homomorphisme $\beta$ de \eqref{eq2.3} est un isomorphisme pour
tout $n\ge 0$.
\end{lemme}

\begin{proof} Dans le diagramme commutatif
\[\begin{CD}
\displaystyle \bigoplus_{x\in X^{(0)}} H^{n+1}_x(X,\Z_l(n)) @>>> \displaystyle\bigoplus_{x\in X^{(1)}} H^{n+2}_x(X,\Z_l(n))\\
 @VVV @VVV\\
\displaystyle \bigoplus_{x\in X^{(0)}} H^{n+1}_x(X,K(n)) @>>> \displaystyle\bigoplus_{x\in X^{(1)}} H^{n+2}_x(X,K(n))
 \end{CD}\]
les deux fl\`eches verticales sont des isomorphismes. En effet, elle s'ins\`erent dans des
suites exactes du type
\[\begin{CD}
H^{n+1}_x(X,\Z(n))\otimes \Z_l &\to& H^{n+1}_x(X,\Z_l(n))&\to& H^{n+1}_x(X,K(n)) &\to& H^{n+2}_x(X,\Z(n))\otimes \Z_l \\
H^{n+2}_x(X,\Z(n))\otimes \Z_l &\to& H^{n+2}_x(X,\Z_l(n))&\to& H^{n+2}_x(X,K(n)) &\to& H^{n+3}_x(X,\Z(n))\otimes \Z_l 
\end{CD}\]
o\`u dans la premi\`ere suite, $x$ est de codimension $0$ et dans la seconde suite, $x$ est de
codimension $1$. Sans perte de g\'en\'eralit\'e, on peut supposer $X$ connexe et alors, pour
son point g\'en\'erique $\eta$ (cf. th\'eor\`eme \ref{tloc} a)):
\[H^i_\eta(X,\Z(n)):=\colim_U H^i(U,\Z(n))\iso  H^i(k(X),\Z(n))=0 \text{ pour } i>n.\]

Pour $x$ de codimension $1$, on a
\[H^i_x(X,\Z(n)) := \colim_{U\ni x} H^i_{Z_U}(U,\Z(n))\]
o\`u $Z_U=\overline{\{x\}}\cap U$. Gr\^ace au th\'eor\`eme \ref{tloc} b) et a), cette limite
devient
\[\colim_{U\ni x} H^{i-2}(Z_U,\Z(n-1))= H^{i-2}(k(x),\Z(n-1)) = 0 \text{ pour } i-2>n-1.\]

Le lemme en d\'ecoule.
\end{proof}

\subsection{Cohomologie \`a supports de $K(n)$} On aura aussi besoin des deux lemmes suivants
au prochain num\'ero:

\begin{lemme}\label{l2.3} Soit $l\ne\car k$.\\
Soit $Y\subset X$ un couple lisse de codimension $d$. Alors il existe des isomorphismes
\[H^{i-2d}_\Zar(Y,K(n-d))\iso H^i_Y(X,K(n)) \quad (n\ge 0, i\in \Z)\]
contravariants pour les immersions ouvertes $U\inj X$.
\end{lemme}

Attention: ce lemme affirme l'existence d'isomorphismes de puret\'e, mais ne dit rien sur leur
caract\`ere canonique ou fonctoriel au-del\`a de la contravariance \'enonc\'ee. (On peut faire
en sorte qu'ils soient contravariants pour les morphismes quelconques entre vari\'et\'es
lisses, mais c'est plus technique et inutile ici.)

\begin{proof} Notons $i$ l'immersion ferm\'ee $Y\to X$. On remarque que le diagramme de
$D(Y_\Zar)$
\[\begin{CD}
\Z(n-d)_Y\otimes \Z_l[-2d] @>\cl^{n-d}_Y>> R\alpha_*\Z_l(n-d)^c_Y[-2d]\\
@V{f}VV @V{g}VV\\
Ri^!_\Zar\Z(n)_X\otimes \Z_l @>\cl^n_X>> R\alpha_*Ri^!_\et\Z_l(n)^c_X
\end{CD}\]
o\`u $f$ est induit par le th\'eor\`eme \ref{tloc} b) et $g$ est donn\'e par le th\'eor\`eme
de puret\'e en cohomologie \'etale, est commutatif: cela r\'esulte tautologiquement de la
construction des classes de cycles motiviques dans \cite{gl}. Par cons\'equent, ce diagramme
s'\'etend en un diagramme commutatif de triangles exacts
\[\begin{CD}
\Z(n-d)_Y\otimes \Z_l[-2d] @>\cl^{n-d}_Y>> R\alpha_*\Z_l(n-d)^c_Y[-2d]&\to& K(n-d)_Y[-2d]&\by{+1}\\
@V{f}VV @V{g}VV @V{h}VV\\
Ri^!_\Zar\Z(n)_X\otimes \Z_l @>\cl^n_X>> R\alpha_*Ri^!_\et\Z_l(n)^c_X&\to& Ri^!_\Zar K(n)_X&\by{+1}
\end{CD}\]

(Rien n'est dit sur un choix privil\'egi\'e de $h$.) Comme $f$ et $g$ sont des
quasi-isomorphismes, $h$ en est un aussi, d'o\`u l'\'enonc\'e.

Comme $h$ est un morphisme dans la cat\'egorie d\'eriv\'ee des faisceaux Zariski sur $Y$, la contravariance annonc\'ee est tautologique pour $U\inj X$ tel que $U\cap Y\ne \emptyset$, et elle est sans contenu lorsque $U\cap Y =\emptyset$.
\end{proof}

\begin{lemme}\label{l2.3p} Soit $l=p=\car k$.\\
Soit $Y\subset X$ un couple lisse de codimension $d$. Alors il existe des homomorphismes
\[H^{i-2d}(Y,K(n-d))\by{h^i} H^i_Y(X,K(n)) \quad (n\ge 0, i\in \Z)\]
contravariants pour les immersions ouvertes $U\inj X$. Ce sont des isomorphismes pour $i\le
n+d$.
\end{lemme}

\begin{proof}  On raisonne comme dans la d\'emonstration du lemme \ref{l2.3}, en
utilisant cette fois le th\'eor\`eme \ref{t1} b) et les r\'esultats de Gros
\cite{gros}. D'apr\`es \cite[(3.5.3) et th. 3.5.8]{gros}, on a
\begin{equation}\label{eq3.3}
 R^qi^! \nu_r(n) =
\begin{cases}
0 & \text{si $q\ne d,d+1$}\\
\nu_r(n-d) &\text{si $q=d$.}
\end{cases}
\end{equation}

La formule \eqref{eq3.3} et sa compatibilit\'e aux classes de cycles motiviques (elle est \`a
la base de leur construction) fournit un diagramme commutatif dans $D(Y_\Zar)$
\begin{equation}\label{eq3.3a}
\begin{CD}
\Z(n-d)_Y\otimes \Z_p@>>> R\alpha_* \Z_p(n-d)^c[-2d]\\
@V{f}VV @V{g}VV\\
Ri^!_\Zar \Z(n)_X\otimes \Z_p@>>> R\alpha_*Ri^!_\et \Z_p(n)^c.
\end{CD}
\end{equation}

Il en r\'esulte un morphisme
\[K(n-d)_Y\by{h} Ri^!_\Zar K(n)_X\]
compl\'etant le carr\'e ci-dessus en un diagramme commutatif de triangles exacts. Ceci fournit les homomorphismes $h^i$ du lemme. 

Dans le diagramme \eqref{eq3.3a}, $f$ est un isomorphisme (th\'eor\`eme \ref{tloc} b)).  Par \eqref{eq3.3}, le c\^one de  $g$ est acyclique en degr\'es $\le n+d$. Par cons\'equent, le c\^one de $h$ est acyclique en degr\'es $\le n+d$, ce qui donne la bijectivit\'e de $h^i$ pour $i\le n+d$.
\end{proof}

\subsection{Fin de la d\'emonstration du th\'eor\`eme \ref{p2et}}\label{s3.2}
 
On prend maintenant $n=2$. Le r\'esultat principal est:

\begin{prop}\label{p2.2} Pour $n=2$, l'homomorphisme $\alpha$ de \eqref{eq2.3} est surjectif de
noyau $A^1(X,H^2(K(2)))$, uniquement divisible.
\end{prop}

\begin{proof}  Notons $E_2^{a,b}=A^a(X,H^b(K(n))$: c'est la cohomologie d'un certain complexe de Cousin.
 
En utilisant les lemmes \ref{l2.3} et \ref{l2.3p}, on trouve que $E_1^{a,b} = 0$ pour 
\begin{description}
\item[$l\ne p$] $a>2$ et $a+b<2a$; $a=2$ et $a+b\le 4$.
\item[$l=p$] $a>2$ et $a+b<2+a$; $a=2$ et $a+b\le 4$. 
\end{description}

(En particulier, $E_2^{2,2} = 0$ puisque $K(0)=0$!) On en d\'eduit une suite exacte
\begin{equation}\label{eq3.10}
0\to A^1(X,H^2(K(2)))\to H^3(X,K(2))\\
\to A^0(X,H^3(K(2)))\to 0.
\end{equation}

Mais $A^1(X,H^2(K(2)))$ est l'homologie du complexe
\[E_1^{0,2} \to E_1^{1,2}\to E_1^{2,2}\]
dont tous les termes sont encore dans le domaine d'application des lemmes \ref{l2.3} et
\ref{l2.3p} (isomorphismes de puret\'e). D'apr\`es le lemme \ref{l2.4}, ils sont uniquement
divisibles, ainsi donc que $E_2^{1,2}$. 
\end{proof}

Le th\'eor\`eme \ref{p2et} r\'esulte maintenant de la proposition \ref{p2.1}, du lemme
\ref{l2.4}, du lemme \ref{l2.2} et de la proposition \ref{p2.2}.\qed

\subsection{Un compl\'ement} Notons pour conclure:

\begin{lemme}\label{l3.4} Le $\Z_l$-module $H^0(X,\sH^3_\cont(\Z_l(2)))$ est sans torsion.
\end{lemme}

\begin{proof} On fait comme dans \cite[th. 3.1]{ct-v} (cet argument remonte \`a Bloch-Srinivas
\cite{b-sri} pour la cohomologie de Betti): le th\'eor\`eme de Merkurjev-Suslin implique que le
faisceau $\sH^3_\cont(\Z_l(2))$ est sans torsion. (Pour $l=p$, utiliser \cite{gl2}.)
\end{proof}

\section{Cas d'un corps de base s\'eparablement clos} Soient $k$ un corps s\'eparablement
clos et $X$ une $k$-vari\'et\'e lisse. On veut pr\'eciser le th\'eor\`eme \ref{p2et} dans ce
cas, toujours dans l'esprit de Colliot-Th\'el\`ene--Voisin \cite{ct-v}.

\subsection{Lien avec les cycles de Tate} Le lemme suivant est d\'emontr\'e dans \cite{ctk}.
Il montre que les cycles de Tate entiers fournissent un bon analogue des cycles de Hodge
entiers consid\'er\'es dans \cite{ct-v}:

\begin{lemme}\label{l3.1} Soient $G$ un groupe profini et $M$ un $\Z_l$-module de type fini
muni d'une action continue de $G$.  Soit 
\[M^{(1)}=\bigcup_U M^U\]
o\`u $U$ d\'ecrit les sous-groupes ouverts de $G$. Alors $M/M^{(1)}$ est sans torsion.
\end{lemme}

On en d\'eduit:

\begin{lemme} Supposons  que $k$ soit la cl\^oture s\'eparable d'un corps de type fini et que
$l\ne \car k$. Alors le groupe fini
$C_\tors$ du th\'eor\`eme \ref{p2et} est aussi le sous-groupe de torsion de
\[\Coker\left(CH^2(X)\otimes\Z_l\to H^4_\cont(X,\Z_l(2))^{(1)}\right).\]
Sous la conjecture de Tate, ce conoyau est enti\`erement de torsion  (pour $X$ non
n\'ecessairement propre, cf. Jannsen \cite[p. 114, th. 7.10 a)]{jannsen2}).\qed
\end{lemme}

On peut d'ailleurs supprimer l'hypoth\`ese que $k$ soit la cl\^oture s\'eparable d'un corps de
type fini. En g\'en\'eral, \'ecrivons $k= \bigcup_\alpha k_\alpha$, o\`u $k_\alpha$ d\'ecrit
l'ensemble (ordonn\'e filtrant) des cl\^otures s\'eparables des sous-corps de type fini de $k$
sur lesquels $X$ est d\'efinie. Pour tout $\alpha$, notons $X_\alpha$ le $k_\alpha$-mod\`ele de
$X$ correspondant. Si $k_\alpha\subset k_\beta$, on a des isomorphismes
\[H^4_\cont(X_\alpha,\Z_l(2))\iso H^4_\cont(X_\beta,\Z_l(2))\iso H^4_\cont(X,\Z_l(2))\]
et on peut d\'efinir
\[H^4_\cont(X,\Z_l(2))^{(1)}:=\colim_\alpha H^4_\cont(X_\alpha,\Z_l(2))^{(1)}.\]

Il s'agit en fait d'une limite inductive d'isomorphismes puisque, si $X$ est d\'efinie sur
$k_\alpha^0\subset k_\alpha$ de type fini et de cl\^oture s\'eparable $k_\alpha$ et que
$k_\beta\supset k_\alpha$, l'homomorphisme $Gal(k_\beta/k_\beta k_\alpha^0)\to
Gal(k_\alpha/k_\alpha^0)$ est surjectif.

\subsection{Lien avec le groupe de Griffiths}

Si $k=\C$ et $X$ est projective, Colliot-Th\'el\`ene et
Voisin ont \'etabli un lien entre $H^0(X,\sH^3(\Q/\Z(2)))$ et le groupe de Griffiths dans
\cite[\S 4.2]{ct-v}. Reprenons cette id\'ee en l'amplifiant un peu.

Voici d'abord une d\'efinition de groupes de Griffiths et de groupes d'\'equivalence
homologique dans le contexte
$l$-adique. Supposons $k$ s\'e\-pa\-ra\-ble\-ment clos si $l\ne \car k$, et $k$
alg\'ebriquement clos si $l=\car k$. Par un argument bien connu de Bloch (\cite[lemma 7.10]{bo}, \cite[Lect. 1, lemma 1.3]{blect}),  pour toute
$k$-vari\'et\'e lisse $X$, le sous-groupe de $CH^n(X)$ form\'e des cycles alg\'ebriquement
\'equivalents \`a z\'ero est
$l$-divisible; les diagrammes commutatifs
\[\begin{CD}
CH^n(X)\otimes \Z_l@>\cl^n>> H^{2n}_\cont(X,\Z_l(n))\\
@VVV @VVV \\
CH^n(X)\otimes \Z/l^s@>\cl^n_s>> H^{2n}_\et(X,\Z/l^s(n))
\end{CD}\]
et l'isomorphisme
\begin{equation}\label{eq4.2}
H^{2n}_\cont(X,\Z_l(n))\iso  \lim_s H^{2n}_\et(X,\Z/l^s(n))
\end{equation}
montrent donc que $\cl^n_s$ et $\cl^n$ se factorisent \`a
travers l'\'equivalence al\-g\'e\-bri\-que. 

(Pr\'ecisons. L'isomorphisme \eqref{eq4.2} est valable  pourvu
que le syst\`eme projectif $(H^{2n-1}_\et(X,\Z/l^s(n)))_{s\ge 1}$ soit de Mittag-Leffler. Pour
$l\ne\car k$ c'est vrai parce que les termes sont finis; pour $l=\car
k$ et $X$ projective c'est expliqu\'e dans
\cite[p. 783]{ctss}, donc il faut a priori supposer $X$ projective dans ce cas.)  

Ceci donne un sens \`a:

\begin{defn}\label{d4.1} Soit $X$ une $k$-vari\'et\'e lisse o\`u $k$ est s\'eparablement
clos si $l\ne \car k$ et alg\'ebriquement clos si $l=\car k$; dans ce dernier cas, on
suppose aussi $X$ projective. Soit $A\in
\{\Z_l,\Q_l,\Z/l^n,\Q_l/\Z_l\}$. On note
\begin{align*}
\Griff^n(X,A) &= \Ker\left(A^n_\alg(X)\otimes A\by{\cl^n}
H^{2n}_\cont(X,A(n))\right)\\
A^n_\hom(X,A) &= \IM\left(A^n_\alg(X)\otimes A\by{\cl^n}
H^{2n}_\cont(X,A(n))\right).
\end{align*}
\end{defn}

\begin{rque}\label{r4.2} Si $k=\C$, on a $\Griff^n(X,\Z_l)=\Griff^n(X)\otimes \Z_l$ par
l'isomorphisme de comparaison entre cohomologies de Betti et $l$-adique, o\`u $\Griff^n(X)$ est
d\'efini \`a l'aide de la cohomologie de Betti.
\end{rque}

On a la version $l$-adique de \cite[th. 7.3]{bo}:

\begin{prop}\label{p3.1} Supposons $l\ne \car k$. Notons $A^n_\alg(X)$ le groupe des cycles de
codimension $n$ sur $X$ modulo l'\'equivalence alg\'ebrique. Dans la suite spectrale de coniveau
\[E_r^{p,q} \Rightarrow H^{p+q}(X,\Z_l(n))\]
 pour la cohomologie $l$-adique de $X$, on a un isomorphisme
\[A^n_\alg(X)\otimes \Z_l\iso E_2^{n,n}\]
induit par l'isomorphisme
\[Z^n(X)\otimes \Z_l\iso E_1^{n,n}\]
donn\'e par les isomorphismes de puret\'e. 
\end{prop} 

\begin{proof} C'est la m\^eme que celle de \cite[preuve du th. 7.3]{bo}, \emph{mutatis
mutandis}. Plus pr\'ecis\'ement, la premi\`ere \'etape est identique. Dans la
deuxi\`eme
\'etape, on remplace la d\'esingularisation
\`a la Hironaka des cycles de codimension $n$ de $X$ par une d\'esingularisation \`a la de Jong
\cite[th. 4.1]{dJ}; pour obtenir des r\'esultats entiers, on utilise le th\'eor\`eme de Gabber
\cite{gabber} disant qu'on peut trouver une telle d\'esingularisation de degr\'e premier \`a
$l$. Enfin, l'argument transcendant de \cite{bo} pour prouver que \'equivalences alg\'ebrique et
homologique co\"\i ncident pour les diviseurs sur une vari\'et\'e lisse $Y$ est remplac\'e par
le suivant: par \cite[lemme 7.10]{bo}, le noyau de $CH^n(Y)\to A^n_\alg(Y)$ est $l$-divisible,
donc les suites exactes de Kummer
\[\Pic(X)\by{l^n} \Pic(X)\to H^2(X,\Z/l^n(1))\]
d\'efinissent des injections
\[0\to A^1_\alg(X)/l^n\to H^2(X,\Z/l^n(1))\]
d'o\`u \`a la limite
\[0\to A^1_\alg(X)\otimes \Z_l\to H^2_\cont(X,\Z_l(1))\]
puisque $A^1_\alg(X)$ est un $\Z$-module de type fini.
\end{proof}

Notons que ces arguments ne n\'ecessitent pas que $X$ soit projective (pour le dernier, cf.
\cite[th. 3]{picfini}). Si le th\'eor\`eme de Gabber n'\'evitait pas $l=p$, la d\'emonstration
s'\'etendrait \`a ce nombre premier.

\begin{conv}
\emph{\`A partir de maintenant, $l$ est suppos\'e diff\'erent de $\car k$ sauf mention
expresse du contraire.} La raison essentielle pour cela est que cette restriction appara\^\i t
dans la proposition
\ref{p3.1} (cf. le commentaire ci-dessus).
\end{conv}

\begin{cor}[cf. \protect{\cite[th. 2.7]{ct-v}}]\label{c3.2} On a une suite
exacte
\[H^3_\cont(X,\Z_l(2))\by{\alpha} H^0(X,\sH^3_\cont(\Z_l(2)))\to \Griff^2(X,\Z_l)\to 0.\]
\end{cor}

\begin{proof} Cela r\'esulte de la suite exacte
\[H^3_\cont(X,\Z_l(2))\to E_2^{0,3} \to E_2^{2,2}\by{c} H^4_\cont(X,\Z_l(2))\]
provenant de la suite spectrale de Bloch-Ogus en poids $2$, de l'identification de $E_2^{0,3}$
\`a $H^0(X,\sH^3_\cont(\Z_l(2)))$, de celle de $E_2^{2,2}$ \`a $A^2_\alg(X)\otimes\Z_l$
(proposition \ref{p3.1}) et de celle de $c$ \`a l'application classe de cycle.
\end{proof}

L'analogue complexe du corollaire suivant devrait figurer dans \cite{ct-v}:

\begin{cor}\label{c4.1} 
a) On a une suite exacte, modulo des groupes finis:
\begin{multline*}
H^3_\cont(X,\Z_l(2))\otimes \Q/\Z \to  H^0(X,\sH^3(\Q_l/\Z_l(2))
\\
\to\Griff^2(X,\Z_l)\otimes \Q/\Z\to 0
\end{multline*}
(cf. d\'efinition \ref{d4.1}).\\
b) Le groupe $H^0(X,\sH^3(\Q_l/\Z_l(2))$ est d\'enombrable.\\
c) Si $\car k = 0$, il existe $X/k$ projective lisse telle que son corang
soit infini pour $l$ convenable. 
\end{cor}

(Pr\'ecisons: ``modulo des groupes finis" signifie ``dans la localisation de la cat\'egorie des groupes ab\'eliens relative \`a la sous-cat\'egorie \'epaisse des groupes ab\'eliens finis".)

\begin{proof} a) r\'esulte du th\'eor\`eme \ref{p2et} et du corollaire \ref{c3.2}. 
Pour b), le terme de gauche dans la suite de a) est de cotype fini, donc d\'enombrable, et le
terme de droite l'est aussi (propri\'et\'e classique des cycles modulo l'\'equivalence
alg\'ebrique). Enfin, d'apr\`es Schoen \cite{schoen}, on a des exemples de $X$ et de nombres
premiers $l$ (m\^eme sur $\bar \Q$) o\`u $\Griff^2(X)/l$ est infini; en utilisant \cite[prop.
4.1]{ct-v} (voir aussi corollaire \ref{c4.10} du pr\'esent article), on en d\'eduit que $\Griff^2(X)\otimes \Q_l/\Z_l$ est de corang infini.
\end{proof}

\subsection{Quelques calculs de groupes de torsion} On veut maintenant pr\'eciser le
corollaire \ref{c4.1} a) en d\'ecrivant explicitement le noyau de l'application
$H^3_\cont(X,\Z_l(2))\otimes \Q/\Z\to H^0(X,\sH^3_\cont(\Z_l(2)))\otimes \Q/\Z$.

\begin{defn} Soit $A$ un $\Z_l$-module de la forme $\Z_l,\Q_l,\Z/l^n,\Q_l/\Z_l$. On note $N
H^3_\cont(X,A)$ le premier cran de la filtration par le coniveau sur
$H^3_\cont(X,A)$ et 
\[H^3_\tr(X,A)= \frac{H^3_\cont(X,A)}{N H^3_\cont(X,A)}.\]
Si on a un twist \`a la Tate, on note $N H^3_\cont(X,A(n)):= N H^3_\cont(X,A)(n)$.
\end{defn}

Notons que $NH^3_\cont(X,A(2))= H^3(X,A(2))$ (cohomologie motivique de Nisnevich) pour
$A=\Z/l^n$ ou $\Q_l/\Z_l$, d'apr\`es la proposition \ref{p3}.

\begin{rque}\label{r4.1} Si $X$ est propre, les $\Z_l$-modules $H^3_\tr(X,A)$ sont des
invariants birationnels de $X$, avec l'action de $Gal(k/k_0)$ si $X$ est d\'efini sur un
sous-corps $k_0$ de cl\^oture s\'eparable $k$. C'est d\^u \`a Grothendieck
\cite[9.4]{BrIII}. 
\end{rque}

Le lemme \ref{l3.4} implique:

\begin{lemme}\label{l3.4a} Le $\Z_l$-module de type fini $H^3_\tr(X,\Z_l(2))$ est sans
torsion.\qed
\end{lemme}

Par d\'efinition de $NH^3_\cont(X,\Z_l(2))$, la suite exacte du corollaire \ref{c3.2} se
raffine en une suite exacte:
\[0\to H^3_\tr(X,\Z_l(2))\to H^0(X,\sH^3_\cont(\Z_l(2)))\to \Griff^2(X,\Z_l)\to 0\]
qui montre incidemment que $\Griff^2(X,\Z_l)$ est un invariant birationnel pour $X$ propre et
lisse (cf. remarque \ref{r4.1}). En r\'eutilisant le lemme \ref{l3.4}, on en d\'eduit:

\begin{prop}\label{l4.2} On a une suite exacte
\begin{multline*}
0\to \Griff^2(X,\Z_l)_\tors\to H^3_\tr(X,\Z_l(2))\otimes \Q/\Z\\
\to H^0(X,\sH^3_\cont(\Z_l(2)))\otimes \Q/\Z\to \Griff^2(X,\Z_l)\otimes \Q/\Z\to 0.\qed
\end{multline*}
\end{prop}

\begin{rque}\label{r4.3} Dans cette remarque, nous adoptons la convention contravariante pour
les motifs purs sur un corps. Supposons que $X$, de dimension $d$, v\'erifie la conjecture
standard C et la conjecture de nilpotence suivante: l'id\'eal $\Ker(CH^d(X\times X)\otimes
\Q\to A^d_\num(X\times X)\otimes \Q)$ de l'anneau des correspondances de Chow est nilpotent.
Ces propri\'et\'es sont v\'erifi\'ees par exemple si $X$ est une vari\'et\'e ab\'elienne:
Lieberman-Kleiman \cite{kleiman} pour la premi\`ere et Kimura \cite{kimura} pour la seconde.
Alors le motif num\'erique de $X$ admet une d\'ecomposition de K\"unneth, qui se rel\`eve en
une d\'ecomposition de Chow-K\"unneth de son motif de Chow:
\[h(X)=\bigoplus_{i=0}^{2d} h^i(X).\]

Mais le th\'eor\`eme de semi-simplicit\'e de Jannsen \cite{jannsenss} implique que chaque
facteur num\'erique
$h^i_\num(X)$ admet une d\'ecomposition plus fine, provenant de sa d\'ecomposition isotypique:
\[h^i_\num(X)=\bigoplus_{j=0}^{i/2} h^{i,j}_\num(X)(-j)\]
o\`u $h^{i,j}_\num(X)$ est effectif mais aucun facteur simple de $h^{i,j}_\num(X)(1)$ n'est
effectif. Cette d\'ecomposition se rel\`eve de nouveau pour donner une d\'e\-com\-po\-si\-tion
de Chow-K\"unneth raffin\'ee (cf. \cite[th. 7.7.3]{kmp}):
\[h(X)=\bigoplus_{i=0}^{2d} \bigoplus_{j=0}^{i/2} h^{i,j}(X)(-j).\]

Notons $\Ab$ la cat\'egorie des groupes ab\'eliens, $\sA$ le quotient de $\Ab$ par la
sous-cat\'egorie \'epaisse des groupes ab\'eliens d'exposant fini et, pour tout anneau
commutatif $R$, $\Chow(k,R)$ la cat\'egorie des motifs de Chow \`a coefficients dans $R$. On
observe que le foncteur 
\[\Hom(-,-):\Chow(k,\Z)^\op\times \Chow(k,\Z)\to \Ab\]
s'\'etend en un foncteur
\[\Hom(-,-):\Chow(k,\Q)^\op\times \Chow(k,\Q)\to \sA.\]

Par construction, $H^3_\tr(X,\Q_l) =H^*_\cont(h^{3,0}(X),\Q_l)$. Au moins si $d=3$, on peut
montrer que, d'autre part,
\[\Griff^2(X,\Z_l)\simeq  \Hom(h^{3,0}(X),\L)\otimes \Z_l\]
o\`u $\L$ est le motif de Lefschetz et l'isomorphisme est dans $\sA$. Ainsi, la proposition
\ref{l4.2} et le th\'eor\`eme \ref{p2et} montrent que (si $d=3$) \emph{la nullit\'e de
$h^{3,0}(X)$ entra\^\i ne la finitude de $H^0(X,\sH^3(\Q_l/\Z_l(2)))$}. 

D'autre part, la conjecture de Bloch-Beilinson--Murre \cite{jannsen} implique que la nullit\'e
de
$h^{3,0}(X)$ d\'ecoule de celle de $H^3_\tr(X,\Q_l)$: conjecturalement, celle-ci est donc
suffisante pour impliquer la finitude du groupe $H^0(X,\sH^3(\Q_l/\Z_l(2)))$ (au moins si
$\dim X=3$). 

Ceci est une variante de la conjecture 4.5 de Colliot-Th\'el\`ene--Voisin \cite{ct-v}. On
verra au th\'eor\`eme \ref{t3.1} c) qu'elle est vraie si $k$ est la cl\^oture alg\'ebrique d'un
corps fini $k_0$ et que $X$ provient de la classe $B_\tate(k_0)$ de \cite{cell}.

On peut se demander si la r\'eciproque est vraie. Elle est fausse, cf. th\'eor\`eme \ref{t5.1}.
\end{rque}

Le lemme \ref{l3.4a} et la proposition \ref{p3} donnent un diagramme commutatif de suites
exactes
\begin{equation}\label{eq4.4}
\begin{CD}
\s 0&\to& \s N H^3_\cont(X,\Z_l(2))\otimes \Q/\Z &\to& \s H^3_\cont(X,\Z_l(2))\otimes \Q/\Z
&\to&\s  H^3_\tr(X,\Z_l(2))\otimes \Q/\Z&\to&\s 0\\ && @V{a}VV @VVV @V{b}VV\\
\s 0&\to&\s  H^3_\Nis(X,\Q_l/\Z_l(2)) &\to&\s  H^3_\et(X,\Q_l/\Z_l(2)) &\to&\s 
H^3_\tr(X,\Q_l/\Z_l(2)) &\to& \s 0
\end{CD}
\end{equation}
dans lequel la fl\`eche verticale centrale est injective de conoyau fini, isomorphe \`a
$H^4_\cont(X,\Z_l(2))_\tors$. Par le lemme du serpent, on en d\'eduit:

\begin{prop}\label{p4.2} Avec les notations de \eqref{eq4.4}, $a$ est injective 
et on a une suite exacte
\[0\to \Ker b \to \Coker a \to H^4_\cont(X,\Z_l(2))_\tors\to \Coker b\to 0.\]
\end{prop}

Voici une application de la proposition \ref{p4.2}. 

\begin{cor}\label{c4.10} Soit $H^3(X,\Q_l/\Z_l(2))^0$ le noyau de la composition
\[H^3(X,\Q_l/\Z_l(2))\to H^3_\et(X,\Q_l/\Z_l(2))\to H^4_\cont(X,\Z_l(2))_\tors.\]
Alors $\IM a\subset H^3(X,\Q_l/\Z_l(2))^0$ (notations de \eqref{eq4.4}) et  on a un isomorphisme
\[
\Griff^2(X,\Z_l)_\tors\iso \Ker b\iso \Coker a^0
\]
o\`u $a^0:N H^3_\cont(X,\Z_l(2))\otimes \Q/\Z\by{a} H^3(X,\Q_l/\Z_l(2))^0$ est l'application induite par $a$.
\end{cor}

\begin{proof} La premi\`ere assertion est \'evidente.  Notons $a^0$ l'application induite: 
la suite exacte de la proposition \ref{p4.2} induit donc un isomorphisme
\[\Ker b\iso \Coker a^0.\]

D'autre part, la suite exacte de la proposition \ref{l4.2} s'ins\`ere dans un diagramme
commutatif de suites exactes
\[\begin{CD}
0\to \Griff^2(X,\Z_l)_\tors&\to&  H^3_\tr(X,\Z_l(2))\otimes \Q/\Z
&\to& H^0(X,\sH^3_\cont(\Z_l(2)))\otimes \Q/\Z\\
&&@V{b}VV @V{c}VV\\
&& H^3_\tr(X,\Q_l/\Z_l(2))
&\to& H^0(X,\sH^3_\cont(\Q_l/\Z_l(2)))
\end{CD}\]

Comme le faisceau $\sH^3_\cont(\Z_l(2))$ est sans torsion, la suite
\begin{multline*}
0\to H^0(X,\sH^3_\cont(\Z_l(2)))\to H^0(X,\sH^3_\cont(\Q_l(2)))\\
\to H^0(X,\sH^3_\cont(\Q_l/\Z_l(2)))\end{multline*}
est exacte, ce qui signifie que $c$ est injective dans le diagramme ci-dessus. On en d\'eduit un isomorphisme
\[ \Griff^2(X,\Z_l)_\tors \iso \Ker b\]
d'o\`u le corollaire.
\end{proof}

\subsection{Les homomorphismes $A^2_\hom(X,\Z_l)\otimes \Q_l/\Z_l\to A^2_\hom(X,\Q_l/\Z_l)$
et  $\Griff^2(X,\Z_l)\otimes \Q_l/\Z_l\to \Griff^2(X,\Q_l/\Z_l)$} On garde les notations de la
d\'efinition \ref{d4.1}.

\begin{prop}\label{p4.4} On a des suites exactes
\begin{multline*}
A^2_\hom(X,\Z_l)_\tors\to \Griff^2(X,\Z_l)\otimes \Q_l/\Z_l\to \Griff^2(X,\Q_l/\Z_l)\\
\to A^2_\hom(X,\Z_l)\otimes \Q_l/\Z_l\to A^2_\hom(X,\Q_l/\Z_l)\to 0
\end{multline*}
\[C_\tors \to A^2_\hom(X,\Z_l)\otimes \Q_l/\Z_l\to A^2_\hom(X,\Q_l/\Z_l)\]
o\`u $C$ est comme dans \eqref{eqjan} (cf. th\'eor\`eme \ref{p2et}). En particulier, 
l'application $\Griff^2(X,\Z_l)\otimes \Q_l/\Z_l\to \Griff^2(X,\Q_l/\Z_l)$ est de noyau et de
conoyau finis.
\end{prop}

\begin{proof} On a un diagramme commutatif de suites exactes
\[\begin{CD}
\s A^2_\hom(X,\Z_l)_\tors &\to& \s \Griff^2(X,\Z_l)\otimes \Q/\Z &\to& \s
A^2_\alg(X,\Z_l)\otimes \Q/\Z  &\to& \s A^2_\hom(X,\Z_l)\otimes \Q/\Z &\to& \s 0\\
&& @VVV @V{\wr}VV @VVV\\
\s 0 &\to& \s \Griff^2(X,\Q_l/\Z_l) &\to& \s A^2_\alg(X,\Q_l/\Z_l) 
&\to& \s A^2_\hom(X,\Q_l/\Z_l) &\to& \s 0
\end{CD}\]
qui donne la premi\`ere suite de la proposition, par application du lemme du serpent. Pour la
seconde, on utilise le diagramme commitatif de suites exactes
\[\begin{CD}
 \left(\frac{H^4_\cont(X,\Z_l(2))}{N^2H^4_\cont(X,\Z_l(2))}\right)_\tors &\to&
A^2_\hom(X,\Z_l)\otimes \Q/\Z &\to&  H^4_\cont(X,\Z_l)\otimes \Q/\Z\\ && @VVV @VVV\\
 0 &\to&  A^2_\hom(X,\Q_l/\Z_l) &\to&  H^4_\et(X,\Q_l/\Z_l(2)) 
\end{CD}\]
en remarquant que la fl\`eche verticale de droite est injective.
\end{proof}

\subsection{Le sous-groupe de torsion de $CH^2(X,\Z_l)_\alg$}Terminons cette analyse de la
torsion en d\'eterminant celle de  $CH^2(X,\Z_l)_\alg$, sous-groupe de
$CH^2(X)\otimes\Z_l$ form\'e des classes de cycles al\-g\'e\-bri\-que\-ment \'equivalentes \`a
z\'ero, lorsque $X$ est propre: voir corollaire \ref{c4.2}. Pour cela nous avons besoin de la
proposition suivante:

\begin{prop}\label{p4.3} Supposons $k$ s\'eparablement clos, $X/k$ propre et lisse et $i <2n$.
Soit $l$ un nombre premier; si $l=\car k$, on suppose $k$ alg\'ebriquement clos.
Alors l'image de l'application cycle 
\eqref{eq1} est \'egale \`a $H^i_\cont(X,\Z_l(n))_\tors$. En particulier, $H^i_\et(X,\Z(n))$
est extension d'un groupe de torsion (fini pour $l\ne p$) par un groupe divisible, et
\[H^i_\et(X,\Z(n))\otimes \Q_l/\Z_l =0.\]
\end{prop}

\begin{proof} 
\'Etant donn\'e le corollaire \ref{c1a}, il suffit de montrer que
\eqref{eq1} a une image de torsion. On reprend les arguments de
Colliot-Th\'el\`ene et Raskind
\cite{ct-r}: d'apr\`es le th\'eor\`eme
\ref{tloc} a) et le
\S
\ref{nb}, on a
\[H^i_\et(X,\Z(n))=\colim_\alpha H^i_\et(X_\alpha,\Z(n))\]
o\`u $X_\alpha$ parcourt un ensemble ordonn\'e filtrant de mod\`eles de $X$ sur des sous-corps
$F_\alpha$ de type fini sur le corps premier. Il suffit donc de savoir que
$H^i(X,\Z_l(n))^G$ est de torsion, o\`u $G$ est le groupe de Galois absolu de
$F_\alpha^{p^{-\infty}}$. On le voit en se ramenant au cas d'un corps de base fini par
changement de base propre et lisse (SGA4 pour $l\ne p$,  Gros-Suwa pour
$l=p$, \cite[p. 590, th. 2.1]{gros-suwa}), o\`u cela r\'esulte
de la d\'emonstration par Deligne de la conjecture de Weil \cite{deligne} pour $l\ne p$ et du
compl\'ement de Katz et Messing
\cite{katz-messing} pour $l=p$.
\end{proof}

\begin{rque} Supposons $n=2$. En utilisant la suite spectrale de coniveau de
la remarque \ref{r1}, on obtient un isomorphisme
$H^i(X,\Z(2))\simeq H^{i-2}(X,\sK_2)$: alors l'\'enonc\'e n'est autre que celui de
Colliot-Th\'el\`ene--Raskind
\cite[th. 1.8 et 2.2]{ct-r} pour $l\ne p$, et de Gros-Suwa \cite[p. 604, cor. 2.2 et p. 605, th.
3.1]{gros-suwa} pour $l=p$.
\end{rque}

\begin{cor}\label{l3.2} Sous ces hypoth\`eses, les homomorphismes
\begin{align*}
H^3(X,\Q_l/\Z_l(2))&\to CH^2(X)\{l\}\\
H^3_\et(X,\Q_l/\Z_l(2))&\to H^4_\et(X,\Z(2))\{l\}
\end{align*}
sont bijectifs.
\end{cor}

\begin{proof} Pour le second, cela r\'esulte de la suite exacte des coefficients
universels et de la proposition
\ref{p4.3} appliqu\'ee pour
$(i,n)=(3,2)$. Pour le premier, m\^eme raisonnement en utilisant le fait que l'homomorphisme
\[H^i(X,\Z(2))\to H^i_\et(X,\Z(2))\]
est bijectif pour $i\le 3$ par \eqref{eqbl} (qui r\'esulte en poids $2$ du th\'eor\`eme de
Merkurjev-Suslin).
\end{proof}

En particulier, on obtient une injection
\begin{equation}\label{eq3.5} 
H^3_\cont(X,\Z_l(2))\otimes\Q_l/\Z_l\Inj H^4_\et(X,\Z(2)).
\end{equation}

\begin{cor} \label{c4.3} Sous les m\^emes hypoth\`eses, soit $N$ le noyau de l'homomorphisme
$H^4_\et(X,\Z(2))\otimes \Z_l\to H^4_\cont(X,\Z_l(2))$. Alors \eqref{eq3.5} induit un
isomorphisme
\[H^3_\cont(X,\Z_l(2))\otimes\Q_l/\Z_l\iso N_\tors.\]
\end{cor}

\begin{cor}\label{c4.2} Sous les m\^emes hypoth\`eses, on a un isomorphisme canonique:
\[CH^2(X,\Z_l)_\alg\{l\} \iso NH^3_\cont(X,\Z_l(2))\otimes \Q/\Z.\]
\end{cor}

\begin{proof} On va r\'eutiliser le complexe $K(2)$ de la d\'efinition  \ref{d3.1}.
Consid\'erons le diagramme commutatif de suites exactes:

\[\begin{CD}
&&&& \s 0\\
&&&& @VVV\\
&&&& \s H^3_\cont(X,\Z_l(2))_\tors\\
&&&& @VVV\\
\s 0&\to& \s NH^3_\cont(X,\Z_l(2))&\to& \s H^3_\cont(X,\Z_l(2))&\to& \s H^0(X,\sH^3_\cont(\Z_l(2)))&\to& \s \Griff^2(X,\Z_l) &\to& \s 0\\
&&@V{\theta}VV @VVV @V{\wr}VV\\
\s 0&\to& \s H^1(X,\sH^2(K(2)))&\to& \s H^3(X,K(2))&\to& \s H^0(X,\sH^3_\cont(K(2))&\to&\s 0\\
&&&& @VVV\\
&&&& \s CH^2(X,\Z_l)_\hom\\
&&&& @VVV\\
&&&& \s 0
\end{CD}\]
o\`u $CH^2(X,\Z_l)_\hom$ est le noyau de la classe de cycle sur $CH^2(X)\otimes \Z_l$.
L'exactitude \`a gauche de la suite verticale d\'ecoule de la proposition \ref{p4.3}, celle de
la premi\`ere suite horizontale du corollaire \ref{c3.2}, la seconde suite horizontale est
\eqref{eq3.10}, enfin  l'isomorphisme vertical est le lemme \ref{l2.2}. La fl\`eche $\theta$ est induite par le diagramme.

Tout d'abord, le lemme \ref{l3.4} implique via ce diagamme que 
\[NH^3_\cont(X,\Z_l(2))_\tors\iso H^3_\cont(X,\Z_l(2))_\tors.\] 

Appliquons maintenant le lemme du serpent aux deux suites exactes horizontales: on obtient un isomorphisme
\[\Ker \theta \iso H^3_\cont(X,\Z_l(2))_\tors\]
et une suite exacte
\[0\to \Coker \theta\to CH^2(X,\Z_l)_\hom\by{\phi} \Griff^2(X,\Z_l)\to 0\]
et on calcule que $\phi$ est la projection naturelle. Finalement on obtient une suite exacte
\[0\to NH^3_\cont(X,\Z_l(2))/\tors \to H^1(X,\sH^2(K(2)))\to CH^2(X,\Z_l)_\alg\to 0.\]
et l'isomorphisme du corollaire d\'ecoule maintenant de la proposition \ref{p2.2} et de la
suite exacte des Tor \`a coefficients $\Q/\Z$.
\end{proof}

\begin{rque} Si $k$ est la cl\^oture alg\'ebrique d'un corps fini, le groupe $CH^n(X,\Z_l)_\alg$
est de torsion pour toute $k$-vari\'et\'e projective lisse $X$ et tout $n\ge 0$ (r\'eduction au
cas d'une courbe par l'argument de correspondances de Bloch, cf. \cite[preuve de la prop.
2.7]{schoen1}). En particulier, le corollaire \ref{c4.2} d\'ecrit le groupe $CH^2(X,\Z_l)_\alg$
tout entier. (Voir aussi \S \ref{s5.3}.)
\end{rque}

\section{Cas d'un corps de base fini et de sa cl\^oture alg\'ebrique} 

\subsection{Cas d'un corps fini}\label{fini}  Soit $k$ un corps fini. Rappelons d'abord la classe $B_\tate(k)$ de \cite[d\'ef. 1 b)]{cell}

\begin{defn} Une $k$-vari\'et\'e projective lisse $X$ est dans $B_\tate(k)$ si
\begin{thlist}
\item Il existe une $k$-vari\'et\'e ab\'elienne $A$ et une extension finie $k'/k$ telles que le motif de Chow de $X_{k'}$ \`a coefficients rationnels soit facteur direct  de celui de $A_{k'}$.
\item $X$ v\'erifie la conjecture de Tate (sur l'ordre des p\^oles de $\zeta(X,s)$ aux entiers $\ge 0$).
\end{thlist}
\end{defn}

On sait montrer qu'\'etant donn\'e (i), (ii) est cons\'equence de (donc \'equivalent \`a) la conjecture de Tate cohomologique pour la cohomologie $l$-adique, pour un nombre premier $l$ donn\'e pouvant \^etre \'egal \`a la caract\'eristique de $k$ (cela r\'esulte de \cite[lemme 1.9]{cell}, cf. \cite[rem. 3.10]{ctk}).

Consid\'erons les notations de la preuve de la
proposition \ref{p2.1}. Si $k$ est fini et si $X\in B_\tate(k)$, $K_\et$ et $C_\et$ sont finis (ibid., th. 3.6 et lemme 3.7), donc
$K=K_\et=C_\et=0$ et  \eqref{eq2a} devient un isomorphisme
\begin{equation}\label{eq3.0}
H^0(X,\sH^3_\et(\Q_l/\Z_l(2)))\iso C.
\end{equation}

En particulier, $H^0(X,\sH^3_\et(\Q_l/\Z_l(2)))$ est fini et $H^0(X,\sH^3_\cont(\Z_l(2)))\otimes \Q/\Z=0$ (th\'eor\`eme \ref{p2et}). En r\'ealit\'e,
m\^eme le grou\-pe $H^0(X,\sH^3_\et(\Q/\Z(2)))$ est fini: cela r\'esulte de la proposition
\ref{p1} et de la g\'en\'eration finie de $H^4_\et(X,\Z(2))$ \cite[cor. 3.8 c) et e)]{cell}.

Conjecturalement, toute vari\'et\'e projective lisse est dans $B_\tate(k)$.

\subsection{Cas de la cl\^oture alg\'ebrique d'un corps fini} Le but de ce num\'ero est de
d\'emontrer:

\begin{thm}\label{t3.1} Soient $k$ la cl\^oture alg\'ebrique d'un corps fini $k_0$, $X_0\in
B_\tate(k_0)$, $X=X_0\otimes_{k_0} k$ et $l\ne \car k$. Alors\\
a) $\Griff^2(X,\Z_l)$ est de torsion.\\
b) On a une suite exacte
\begin{multline*}
0\to \Griff^2(X,\Z_l)\to H^3_\tr(X,\Z_l(2))\otimes \Q/\Z\\
\to H^0(X,\sH^3_\cont(\Z_l(2)))\otimes\Q/\Z\to 0.
\end{multline*}
c) Si $H^3_\tr(X,\Q_l(2))=0$, le groupe $H^0(X,\sH^3(\Q_l/\Z_l(2)))$ est fini; dans ce cas, il
est isomorphe \`a $C_\tors$.
\end{thm}

\begin{rque}  Le corollaire \ref{c4.10} donne une autre description de
$\Griff^2(X,\Z_l)$.
\end{rque}

\begin{proof} a) $\Rightarrow$ b) par la proposition \ref{l4.2} et b) $\Rightarrow$ c) par le
th\'eor\`eme
\ref{p2et}. Montrons a). Soit $k_1$ une extension finie de
$k_0$, et $X_1=X_0\otimes_{k_0} k_1$. D'apr\`es \cite[th. 3.6 et cor. 3.8 e)]{cell},
l'homomorphisme 
\[H^4_\et(X_1,\Z(2))\otimes \Z_l\to H^4_\cont(X_1,\Z_l(2))\]
est bijectif. D'apr\`es \eqref{eqlk}, l'homomorphisme
\[CH^2(X_1)\otimes\Z_l\to H^4_\cont(X_1,\Z_l(2))\]
est donc injectif. Or dans la suite exacte
\begin{multline}\label{eq3.2}
0\to H^1(G_1,H^3_\cont(X,\Z_l(2)))\to H^4_\cont(X_1,\Z_l(2))\\
\to H^4_\cont(X,\Z_l(2))^{G_1}\to 0
\end{multline}
(o\`u $G_1=Gal(k/k_1)$), le groupe de gauche est fini d'apr\`es Weil I \cite{deligne}. Il en
r\'esulte que le noyau de
\[CH^2(X_1)\otimes\Z_l\to H^4_\cont(X,\Z_l(2))\]
est fini pour tout $k_1$, d'o\`u la conclusion en passant \`a la limite.
\end{proof}

\subsection{Un exemple de Schoen}\label{s5.3} J'avais initialement pens\'e que la r\'eciproque
du th\'eor\`eme \ref{t3.1} c) est vraie. En r\'ealit\'e elle est fausse: cela r\'esulte d'un calcul de C.
Schoen \cite{schoen1}. Dans cet article, Schoen consid\`ere $X=E^3$ sur $k=\bar \F_p$, o\`u $E$
est la courbe elliptique d'\'equation $x^3+y^3+z^3=0$, et montre que, si $p\equiv 1\pmod{3}$:
\[\Griff^2(X)\{l\} \simeq (\Q_l/\Z_l)^2\]
pour $l\equiv -1\pmod{3}$ \cite[th. 0.1]{schoen1}. Le groupe $\Griff^2(X)$ est d\'efini comme
le quotient du groupe des cycles \`a coefficients entiers qui sont homologiquement
\'equivalents \`a z\'ero par le sous-groupe de ceux qui sont al\-g\'e\-bri\-que\-ment
\'equivalents \`a z\'ero. Commen\c cons par clarifier le lien entre ce groupe et
$\Griff^2(X,\Z_l)$:

\begin{prop}\label{p5.1} Soient $k_0$ un corps fini de cl\^oture alg\'ebrique $k$, $X_0\in
B_\tate(k_0)$ et
$X=X_0\otimes_{k_0} k$. Alors, pour tout $n\ge 0$, l'homomorphisme
\'e\-vi\-dent
\[\Griff^n(X)\otimes \Z_l\to \Griff^n(X,\Z_l)\]
est bijectif.
\end{prop}

\begin{proof} Soit $A$ un groupe ab\'elien quelconque. Pour une vari\'et\'e lisse $X$ sur un
corps quelconque, on peut d\'efinir les cycles \`a coefficients dans $A$ modulo l'\'equivalence
rationnelle, ou alg\'ebrique. Notons ces groupes $CH^*(X,A)$ et $A^*_\alg(X,A)$. Je dis que les
homomorphismes
\begin{align*}
CH^*(X)\otimes A &\to CH^*(X,A)\\
A^*_\alg(X)\otimes A &\to A^*_\alg(X,A)
\end{align*}
sont bijectifs: par exemple on peut d\'ecrire
$A^n_\alg(X,A)$ comme le conoyau d'un homomorphisme
\[\bigoplus_{(V,v_0,v_1)}\bigoplus_{x\in ((X\times V)^{(n)})'}  A\by{s_0^*-s_1^*}
\bigoplus_{x\in X^{(n)}} A\]
o\`u $(V,v_0,v_1)$ d\'ecrit l'ensemble des classes d'isomorphismes de $k$-vari\'et\'es lisses
$V$ munies de deux points rationnels $v_0$ et $v_1$.

Pla\c cons-nous maintenant dans la situation de la proposition. Notons $CH^n(X)_\hom$ le
noyau de
$\cl^n:CH^n(X)\to H^{2n}_\cont(X,\Z_l(n))$. Je dis que l'isomorphisme
\[CH^n(X)\otimes \Z_l\iso CH^n(X,\Z_l)\]
envoie $CH^n(X)_\hom\otimes \Z_l$ sur $CH^n(X,\Z_l)_\hom$. En effet, soit  $x\in\allowbreak
CH^n(X,\Z_l)_\hom$. \'Ecrivant $x=\sum \alpha_i x_i$ avec $\alpha_i\in \Z_l$, $x_i\in
CH^n(X)$, on peut (quitte \`a augmenter $k_0$) supposer que $x$ provient de
$x_0\in CH^n(X_0,\Z_l)$. On a \'evidemment $x_0\in CH^n(X_0,\Z_l)_\hom$; le m\^eme raisonnement
que dans la preuve du th\'eor\`eme \ref{t3.1} (utilisant le fait que $X_0\in
B_\tate(k_0)$) montre alors que $x_0$ est de torsion. Mais, pour
tout groupe ab\'elien $M$, on a des isomorphismes
\[M\{l\}\iso M\{l\}\otimes \Z_l\iso (M\otimes \Z_l)\{l\}\]
puisque $(M/M\{l\})\otimes \Z_l$ est sans $l$-torsion. Donc $x_0\in CH^n(X_0)_\hom$ et $x\in
CH^n(X)_\hom$.

Il r\'esulte de ceci que l'homomorphisme induit
\[\Griff^n(X)\otimes \Z_l=A^n_\alg(X)_\hom\otimes \Z_l\to
A^n_\alg(X,\Z_l)_\hom=\Griff^n(X,\Z_l)\] 
est surjectif, donc bijectif, d'o\`u l'\'enonc\'e.
\end{proof}

\begin{prop} \label{p5.2} Sous les hypoth\`eses de la proposition \ref{p5.1}, les conditions
suivantes sont \'equivalentes:
\begin{thlist}
\item $H^0(X,\sH^3(\Q_l/\Z_l(2)))$ est fini.
\item Le monomorphisme $\Griff^2(X,\Z_l)\to H^3_\tr(X,\Z_l(2))\otimes \Q_l/\Z_l$ du
th\'e\-o\-r\`e\-me
\ref{t3.1} b) est surjectif.
\item $\corg \Griff^2(X,\Z_l)\ge \dim  H^3_\tr(X,\Q_l(2))$.
\item L'application $b$ de \eqref{eq4.4} est nulle.
\item $H^3_\tr(X,\Q_l/\Z_l(2))$ est fini, quotient de $H^4_\cont(X,\Z_l)_\tors$.
\end{thlist}
Si $X$ est une vari\'et\'e ab\'elienne (il suffit que $H^3(X,\Z_l)$ et $H^4(X,\Z_l)$ soient
sans torsion), ces conditions sont encore \'equivalentes \`a:
\begin{thlist}
\item[\rm (vi)] Pour tout $n\ge 1$, l'homomorphisme $H^3_\et(X,\Z/l^n)\to H^3_\et(k(X),\Z/l^n)$
est \emph{nul}.
\end{thlist}
\end{prop}

\begin{proof} Les \'equivalences (i) $\iff$ (ii) $\iff$ (iii) r\'esultent du th\'eor\`eme
\ref{t3.1} b) et du th\'eor\`eme \ref{p2et}. Les \'equivalences (ii) $\iff$ (iv) $\iff$ (v)
r\'esultent du th\'eor\`eme
\ref{t3.1} a), du corollaire \ref{c4.10} et de la proposition \ref{p4.2} (ou plus directement
du diagramme \eqref{eq4.4}). 

Si (vi) est vrai, il est vrai stablement (c'est-\`a-dire \`a coefficients $\Q_l/\Z_l$), ce
qui est \'equivalent \`a la nullit\'e de $H^3_\tr(X,\Q_l/\Z_l)$, d'o\`u (iv).
 R\'e\-ci\-pro\-que\-ment, montrons que (v) $\Longrightarrow$ (vi) si $H^3_\cont(X,\Z_l)$ et
$H^4_\cont(X,\Z_l)$ sont sans torsion. De (v) on d\'eduit que
$H^3_\tr(X,\Q_l/\Z_l(2))=0$, ce qui donne (vi) stablement. Pour l'obtenir \`a
coefficients finis, consid\'erons le diagramme commutatif aux lignes exactes:
\[\begin{CD}
0 &\to& H^2_\Nis(X,\Q_l/\Z_l(2))/l^n &\to& H^3_\Nis(X,\Z/l^n(2)) &\to&
{}_{l^n}H^3_\Nis(X,\Q_l/\Z_l(2))&\to& 0\\
&& @VVV @VVV @V{\wr}VV \\
0 &\to& H^2_\et(X,\Q_l/\Z_l(2))/l^n &\to& H^3_\et(X,\Z/l^n(2)) &\to&
{}_{l^n}H^3_\et(X,\Q_l/\Z_l(2))&\to& 0.
\end{CD}\]

Comme $H^3_\cont(X,\Z_l)$ est sans torsion, $H^2_\et(X,\Q_l/\Z_l(2))$ est divisible et le terme
en bas \`a gauche est nul. La fl\`eche verticale centrale est donc surjective, ce qui donne
l'\'enonc\'e pour $H^3_\et(X,\Z/l^n(2))$.
\end{proof}

\begin{thm}\label{t5.1} Soient $p$ un nombre premier $\equiv 1\pmod{3}$, $k=\bar \F_p$, et $E$
la courbe elliptique sur $k$ d'\'equation $x^3+y^3+z^3=0$. Posons $X=E^3$. Si $l\equiv
-1\pmod{3}$, les conditions de la proposition \ref{p5.2} sont v\'erifi\'ees.
\end{thm}

\begin{proof} Pour commencer, observons que $X\in B_\tate(\F_p)$. Cela r\'esulte du th\'eor\`eme de Spiess \cite{spiess}, ou simplement de Soul\'e \cite[th. 3]{soule} puisque $\dim X = 3$.

Montrons (iii). D'apr\`es
Schoen
\cite[th. 0.1]{schoen1} et la proposition
\ref{p5.1}, on a
$\Griff^2(X,\Z_l)\simeq (\Q_l/\Z_l)^2$;  il
faut donc montrer que $H^3_\tr(X,\Z_l)$ est de rang $\le 2$. Comme $X$ est une vari\'et\'e
ab\'elienne, on a un isomorphisme 
\[\Lambda^3 H^1_\cont(X,\Q_l)\iso H^3_\cont(X,\Q_l).\]

L'hypoth\`ese sur $p$ assure que $E$ est ordinaire (cf. \cite[p. 46]{schoen1}).
Soient $\alpha,
\beta$ les nombres de Weil de $E$ sur $\F_p$: on a
$\alpha\beta=p$, et $K:=\Q(\alpha)=\Q(\mu_3)$ (ibid.). 

L'espace vectoriel $H^1_\cont(X,\Q_l)$ est somme de trois exemplaires de
\allowbreak $H^1_\cont(E,\Q_l)$: il est donc de rang $6$. Soit
$(v_1,v_2,v_3,w_1,w_2,w_3)$ une base du $\Q_l\otimes K$-module libre $H^1_\cont(X,\Q_l)\otimes
K$ form\'ee de vecteurs propres pour l'action du Frobenius $\phi$, avec $\phi v_i=\alpha v_i$,
$\phi w_i=\beta w_i$. Le
$\Q_l\otimes K$-module $H^3_\cont(X,\Q_l)\otimes K$ est libre de rang $20$, de base les
tenseurs purs de degr\'e $3$ construits sur les $v_i,w_j$. Par construction, cette base $B$ est
form\'ee de vecteurs propres pour l'action de Frobenius. 

Soit $b\in B$. Si $b\notin\{v_1\wedge v_2\wedge v_3,w_1\wedge w_2\wedge w_3\}$, $b$ est
divisible par $c=v_i\wedge w_j$ pour un couple $(i,j)$. La valeur propre de $c\in
H^2_\cont(X,\Q_l)\otimes K$ est \'egale \`a $p$; en particulier, $c\in H^2_\cont(X,\Q_l)$. Par
le th\'eor\`eme de Tate (d\^u dans ce cas particulier \`a Deuring), $c\otimes \Q_l(1)\in
H^2_\cont(X,\Q_l(1))$ est de la forme $\cl(\gamma)$ pour un diviseur $\gamma\in \Pic(X)\otimes
\Z_l$: il en r\'esulte que $b\in NH^3_\cont(X,\Q_l)$.

Ceci montre que $H^3_\tr(X,\Q_l)\otimes K$ est engendr\'e par $b=v_1\wedge v_2\wedge v_3$ et
$b'= w_1\wedge w_2\wedge w_3$, et donc que $\dim H^3_\tr(X,\Q_l)\le 2$.
\end{proof}

\begin{rque} 
Comme $H^*_\cont(X,\Z_l)\to
H^*_\et(X,\Z/l^n)$ est surjectif, le calcul fait dans la preuve du th\'eor\`eme \ref{t5.1}
montre a priori que l'image de $H^3_\et(X,\Z/l^n)$ dans $H^3(F,\Z/l^n)$ est de rang $\le 2$,
o\`u $F=k(X)$. On aimerait bien d\'emontrer sa nullit\'e  (l'\'enonc\'e (vi) de la proposition
\ref{p5.2}) directement: il s'agit de voir que,
si
$x_1,x_2,x_2\in H^1_\et(X,\Z/l^n)$, le cup-produit $x_1\cdot x_2\cdot x_3$ est nul dans
$H^3_\et(F,\Z/l^n)$.

On peut se limiter aux triplets $(x_1,x_2,x_3)$ tels que 
\[x_i\in
\IM\big(H^1_\et(E,\Z/l^n)\allowbreak\by{\pi_i^*} H^1_\et(X,\Z/l^n)\big)\] 
pour une valeur de
$i$, o\`u
$\pi_i$ est la $i$-\`eme projection, et sans perte de g\'en\'eralit\'e, supposer
$i=1$. Alors $x_1$ d\'efinit une isog\'enie $f:E'\to E$ de degr\'e $l^n$. Soit $F'=k(X')$,
o\`u $X'=E'\times E\times E$: d'apr\`es Merkurjev-Suslin, la nullit\'e de $x_1\cdot
x_2\cdot x_3$ dans $H^3_\et(F,\Z/l^n)$ \'equivaut au fait que $x_2\cdot x_3\in
H^2(F,\Z/l^n)\simeq {}_{l^n}Br(F)$ est une norme dans l'extension $F'/F$. Peut-on montrer ceci
directement?
\end{rque}

\subsection{Autres corps} Si $cd(k)\le 1$, la suite exacte \eqref{eq3.2} persiste 
\cite[th. 3.3]{jannsen0}.
Malheureusement, elle ne semble pas apporter d'informations suppl\'ementaires tr\`es utiles,
sauf peut-\^etre dans le cas d'un corps quasi-fini que je n'ai pas explor\'e.

Consid\'erons les notations de la preuve de la proposition \ref{p2.1}. Si $k$ est de type fini
mais n'est pas fini, je ne sais pas s'il faut esp\'erer que $K$ est de torsion, m\^eme sous
toutes les conjectures habituelles (Jannsen le sugg\`ere dans \cite[lemma 2.7]{jannsen}).  On
peut remplacer
$H^4_\cont(X,\Z_l(2))$ par le groupe plus fin 
\[\tilde H^4_\cont(X,\Z_l(2))=\colim H^4_\cont(\sX,\Z_l(2))\] 
o\`u $\sX$ d\'ecrit les mod\`eles r\'eguliers de $X$, de type fini sur $\Spec \Z$ (cf.
\cite[(11.6.1)]{jannsen2}). En caract\'eristique $p$,  par passage \`a la limite, la conjecture
de Tate-Beilinson implique alors que $K_\et$ est  de torsion \cite[th. 60, (iii)]{kcag}. De
plus, cette conjecture implique que $CH^2(X)$ est de type fini (comme quotient de $CH^2(\sX)$
pour un mod\`ele $\sX$ lisse de type fini), donc que
$K$ est fini. Par contre, elle n'implique pas a priori que $H^4_\et(X,\Z(2))$ est de type fini
(dans les suites exactes de Gysin pour un diviseur, le terme suivant est de la forme
$H^3_\et(Z,\Z(1)) = Br(Z)$\dots) donc il se pourrait fort bien que $K_\et$ ait une partie
divisible non triviale.

Le bon objet avec lequel travailler pour des vari\'et\'es ouvertes sur un corps fini est
$H^4_W(\sX,\Z(2))$ (cohomologie Weil-\'etale): c'est celui qui permet d'attraper tout $\tilde
H^4_\cont(X,\Z_l(2))$ pour $l\ne p$ \cite[th. 64]{kcag}. Mais cela a l'air compliqu\'e,
cf. \cite[(3.2)]{cell} ou \cite[th. 62 (ii)]{kcag}.

\appendix

\section{Cohomologie de Hodge-Witt logarithmique sur des corps imparfaits} Dans \cite{gl2},
Geisser et Levine comparent la cohomologie motivique modulo $p$ d'un corps de caract\'eristique
$p$ quelconque avec sa cohomologie de Hodge(-Witt) logarithmique, mais n'en d\'eduisent une
comparaison globale que pour des vari\'et\'es lisses sur un corps parfait. Le but de ce
num\'ero est de rappeler les bases de cette comparaison et de se d\'ebarrasser de mani\`ere
``triviale" de l'hypoth\`ese de perfection, \`a l'aide d'une observation classique de Quillen
\cite[p. 133, d\'em. du th. 5.11]{quillen}.

\subsection{Cohomologie de Hodge-Witt logarithmique} Soit $X$ un sch\'ema de caract\'eristique
$p$. On lui associe son pro-complexe de de Rham-Witt
\cite[p. 548, 1.12]{illusie}
\[(W_r\Omega_X^\cdot)_{r\ge 1}\]
qui est un syst\`eme projectif de faisceaux d'alg\`ebres diff\'erentielles gradu\'ees
sur $X_\et$, prolongeant (pour $\cdot=0$) le profaisceau des vecteurs de Witt et (pour $r=1$)
le complexe des diff\'erentielles de K\"ahler. Il est muni d'un op\'erateur $F:W_r\Omega^n_X\to
W_{r-1}\Omega^n_X$ \cite[p. 562, th. 2.17]{illusie}.

Si $X$ est d\'efini sur un corps parfait $k$, on a \'evidemment
\[W_r\Omega_X^\cdot=W_r\Omega_{X/k}^\cdot.\]

On a des applications ``de Teichm\"uller" (multiplicatives)
\begin{gather*}
\sO_X\to W_r\sO_X\\
x\mapsto \underline{x} = (x,0,\dots,0,\dots)
\end{gather*}
\cite[p. 505, (1.1.7)]{illusie}, qu'on utilise pour d\'efinir les homomorphismes
\begin{align}
d\log:\sO_X^*/\left(\sO_X^*\right)^{p^r}&\to W_r\Omega^1_X\label{edlog}\\
x&\mapsto d\underline{x}/\underline{x}\notag
\end{align}
\cite[p. 580, (3.23.1)]{illusie}. On d\'efinit alors $W_r\Omega^n_{X,\log}$ comme le
sous-faisceau de $W_r\Omega^n_X$ engendr\'e localement pour la topologie \'etale par les
sections de la forme
$d\log x_1\wedge\dots\wedge d\log x_n$ \cite[p. 596, (5.7.1)]{illusie};
comme dans \cite{gl2}, nous noterons simplement ce faisceau $\nu_r(n)_X$.

\begin{lemme}\label{lillusie} Pour tout $x\in \Gamma(X,\sO_X)$ et pour tout $r\ge 1$, on a
\[\underline{x}\wedge \underline{1-x} =0\in \Gamma(X,W_r\Omega^2_X).\]
\end{lemme}

\begin{proof} (Illusie) Le morphisme $X\to \A^1_{\F_p}$ d\'efini par $x$ nous ram\`ene au cas
universel $X=\Spec \F_p[t]$, $x=t$. Mais alors $W_r\Omega^2_X=0$ puisque $\dim X = 1$.
\end{proof}

\subsection{Le symbole logarithmique} Supposons $X=\Spec k$, o\`u $k$ est un corps. Le lemme
\ref{lillusie} implique que l'homomorphisme $d\log$ de \eqref{edlog} induit un \emph{symbole
logarithmique}
\begin{align} 
d\log: K_n^M(k)/p^r&\to \nu_r(n)_k\label{edlogM}\\
\{x_1,\dots,x_n\} &\mapsto d\log(x_1)\wedge\dots\wedge d\log(x_n).\notag
\end{align}

Soit $K$ le corps des fonctions d'un $k$-sch\'ema lisse $X$,
o\`u $k$ est parfait de caract\'eristique $p$. Un point $x$ de codimension $1$ de $X$ d\'efinit
une valuation discr\`ete $v$ sur $K$, de corps r\'esiduel $E=k(x)$. Le th\'eor\`eme de puret\'e
de Gros
\cite[p. 46, th. 3.5.8]{gros} et la longue suite exacte de cohomologie \`a supports
d\'efinissent des homomorphismes r\'esidus
\begin{equation}\label{ereslog}
\nu_r(n)_K\by{\partial_v} \nu_r(n-1)_E.
\end{equation}

\begin{lemme}\label{lresid} Le diagramme
\[\begin{CD}
K_n^M(K)/p^r @>\partial_v>> K_{n-1}^M(E)/p^r\\
@V{d\log}VV @V{d\log}VV\\
\nu_r(n)_K@>{\partial_v}>> \nu_r(n-1)_E
\end{CD}\]
o\`u la fl\`eche horizontale du haut est le r\'esidu en $K$-th\'eorie de Milnor, est
commutatif au signe pr\`es.
\end{lemme}

\begin{proof} Pour $n=1,2$ c'est fait dans Gros-Suwa \cite[p. 625, lemme 4.11]{gros-suwa1}. La
d\'emonstration ne se propage pas tout \`a fait \`a $n>2$ car elle utilise la formule explicite
donnant $\partial(\{x,y\})$ pour $x,y\in K^*$. Pour la propager, il suffit toutefois de
remarquer que $K_n^M(K)$ est engendr\'e par les symboles de la forme $\{u_1,\dots,u_{n-1},x\}$
avec $u_i\in O_v^*$ et $x\in K^*$. 
\end{proof}

\subsection{Le morphisme de comparaison} Supposons $X$ lisse sur un corps parfait $k$. D'apr\`es
\cite[p. 597, th. 5.7.2]{illusie}, on a une suite exacte de pro-faisceaux \'etales
\[0\to \nu_\cdot(n)_X\to W_\cdot\Omega^n_X\by{1-F} W_\cdot\Omega^n_X\to 0\]
qui en fait n'interviendra pas ici. De plus, on a le th\'eor\`eme suivant, d\^u \`a Gros et
Suwa:

\begin{thm}\label{tgs} On a une suite exacte de faisceaux zariskiens
\[0\to \alpha_* \nu_r(n)_X\to \bigoplus_{x\in X^{(0)}}
\left(\nu_r(n)_{k(x)}\right)_{\overline{\{x\}}}\by{\partial} \bigoplus_{x\in X^{(1)}}
\left(\nu_r(n-1)_{k(x)}\right)_{\overline{\{x\}}}\by{\partial}\dots\] 
o\`u $\alpha$ d\'esigne la projection $X_\et\to X_\Zar$ et les diff\'erentielles $\partial$
sont cons\-trui\-tes \`a partir des r\'esidus \eqref{ereslog}.
\end{thm}

\begin{proof} Voir \cite[cor. 1.6]{gros-suwa1} ou \cite[p. 70, Ex. 7.4 (3)]{cthk}.
\end{proof}

Supposons maintenant $X$ r\'egulier de type fini sur un corps $k$ (de caract\'eristique $p$).
Supposons d'abord $k$ de type fini sur $\F_p$: alors $X$ admet un mod\`ele $\sX$
r\'egulier, donc lisse, de type fini sur $\F_p$. Soit $j:X\to \sX$ la pro-immersion ouverte
correspondante: on a \'evidemment
\[\nu_r(n)_X = j^*\nu_r(n)_\sX\]
puisque les anneaux semi-locaux de $X$ sont certains anneaux semi-locaux de $\sX$. 

Interpr\'etons maintenant $K_n^M(K)/p^n$ comme $H^n(K,\Z/p(n))$, cf. th\'e\-o\-r\`e\-me
\ref{tK}. Vu la remarque \ref{r1} et le th\'eor\`eme \ref{tgs}, les
homomorphismes \eqref{edlogM} et le lemme
\ref{lresid} induisent des homomorphismes de faisceaux
\[\sH^n(\Z/p^r(n)_\sX)\to \alpha_*\nu_r(n)_\sX\]
et donc des morphismes dans $D^-(\sX_\Zar)$
\[
\Z/p^r(n)_\sX\to \alpha_*\nu_r(n)_\sX[-n]
\]
puisque $\sH^i(\Z/p^r(n)_\sX)=0$ pour $i>n$ (lemme \ref{l2.1}).

D'o\`u, en appliquant $j^*$, des morphismes dans $D^-(X_\Zar)$
\begin{equation}\label{eqcomp}
\Z/p^r(n)_X\to \alpha_*\nu_r(n)_X[-n].
\end{equation}

Si $k$ est
quelconque, \'ecrivons
\begin{gather*}
k=\colim_\alpha k_\alpha\\
X=\lim_\alpha X_\alpha
\end{gather*}
o\`u les $k_\alpha$ sont de type fini sur $\F_p$ et $X_\alpha$ est un $k_\alpha$-sch\'ema
r\'egulier de type fini, de sorte que $k_\alpha \subset k_\beta$ induise un
isomorphisme $X_\beta\iso X_\alpha\otimes_{k_\alpha} k_\beta$. On a \'evidemment:
\begin{align*}
\nu_r(n)_X &= \colim_\alpha \pi_\alpha^*\nu_r(n)_{X_\alpha}\\
\sH^n(\Z/p^r(n)_X) &= \colim_\alpha \pi_\alpha^*\sH^n(\Z/p^r(n)_{X_\alpha})
\end{align*}
o\`u $\pi_\alpha:X\to X_\alpha$ est le morphisme canonique. Ceci \'etend la d\'efinition de
\eqref{eqcomp} au cas o\`u le corps de base est quelconque. On voit de m\^eme:

\begin{prop}[cf. Quillen \protect{\cite[p. 133, d\'em. du th. 5.11]{quillen}}]\label{pquillen}
La suite exacte du th\'eor\`eme \ref{tgs} s'\'etend
\`a tout
$X$ r\'egulier de type fini sur un corps.\qed
\end{prop}

\subsection{Le th\'eor\`eme de Geisser-Levine}

\begin{thm}\label{tgl} Soit $X$ un sch\'ema r\'egulier de type fini sur un corps $k$ de
caract\'eristique $p$. Alors le morphisme \eqref{eqcomp} est un isomorphisme.
\end{thm}

\begin{proof} Il s'agit de voir que
\[\sH^i(\Z/p^r(n)_X)\simeq
\begin{cases}
0 &\text{si $i\ne n$}\\
\alpha_*\nu_r(n)_X &\text{si $i=n$}
\end{cases}
\]
le dernier isomorphisme \'etant induit par \eqref{eqcomp}. L'\'enonc\'e est clair pour $i>n$,
cf. lemme \ref{l2.1}.

1) $X=\Spec k$:  c'est le th\'eor\`eme de Bloch-Gabber-Kato pour $i=n$ \cite[p. 117, cor.
2.8]{bk} et celui de Geisser-Levine \cite[th. 1.1]{gl2} pour $i<n$.

2) $X$ lisse sur $k$ parfait: on se r\'eduit \`a 1) en utilisant le th\'eor\`eme \ref{tgs}, le
lemme \ref{lresid} et la conjecture de Gersten pour la cohomologie motivique, cf. preuve du
lemme \ref{l2.1}.

3) $k$ de type fini sur $\F_p$: on se ram\`ene \`a 2) par la technique du num\'ero
pr\'ec\'edent.

4) $k$ quelconque: on se ram\`ene \`a 3) par passage \`a la limite.
 \end{proof}

\begin{rque} On pourrait court-circuiter les \'etapes 2) et 3), dans l'esprit de la proposition
\ref{pquillen}.
\end{rque}

\begin{cor}\label{cgl} Soit $X$ un sch\'ema r\'egulier de type fini sur un corps $k$ de
caract\'eristique $p$. Alors le morphisme $\alpha^*$\eqref{eqcomp} est un isomorphisme, o\`u
$\alpha$ est la projection $X_\et\to X_\Zar$.\qed
\end{cor}

\end{document}